\documentclass[a4paper, 11pt]{amsart}

\usepackage{amsmath,amssymb,amscd,xypic}
\usepackage{amsfonts}
\usepackage[english]{babel}
\usepackage[leqno]{amsmath}
\usepackage{amssymb,amsthm}
\usepackage{mathrsfs} 
\usepackage{amscd}
\usepackage{enumerate}
\usepackage{epsfig}
\usepackage{relsize}
\usepackage{layout}
\usepackage{fullpage}
\usepackage[usenames,dvipsnames]{xcolor}
\usepackage[backref=page]{hyperref}
\hypersetup{
 colorlinks,
 citecolor=Green,
 linkcolor=Red,
 urlcolor=Blue}
\usepackage[matrix,arrow,tips,curve]{xy}
\input{xy}
\xyoption{all}

\UseRawInputEncoding

\newtheorem{theorem}{Theorem}[section]
\newtheorem{lemma}[theorem]{Lemma}

\newtheorem{proposition}[theorem]{Proposition}
\newtheorem{cor}[theorem]{Corollary}

\newtheorem{fact}[theorem]{Fact}
\newtheorem{question}{Question}

\numberwithin{equation}{section}
\theoremstyle{definition}
\newtheorem{definition}[theorem]{Definition}

\theoremstyle{remark}
\newtheorem{remark}[theorem]{Remark}
\newtheorem{example}[theorem]{Example}

\newenvironment{sis}{\left\{\begin{aligned}}{\end{aligned}\right.}

\newcommand{\bbR}{{\mathbb R}}

\newcommand{\bbQ}{{\mathbb Q}}
\newcommand{\bbZ}{{\mathbb Z}}

\renewcommand{\cD}{{\mathcal D}}

\newcommand{\cO}{{\mathcal O}}

\newcommand{\un}{\underline}
\newcommand{\ov}{\overline}
\newcommand{\wt}{\widetilde}
\newcommand{\wh}{\widehat}
\newcommand{\m}{\mathcal}

\newcommand{\irr}{\operatorname{irr}}
\newcommand{\type}{\operatorname{type}}

\newcommand{\Spf}{\operatorname{Spf}}
\newcommand{\Proj}{\operatorname{Proj}}
\newcommand{\Aut}{\operatorname{Aut}}
\newcommand{\Pic}{\operatorname{Pic}}

\newcommand{\Cl}{\operatorname{Cl}}
\newcommand{\D}{\operatorname{\Delta}}
\newcommand{\id}{\operatorname{id}}
\newcommand{\Def}{\operatorname{Def}}
\newcommand{\ps}{\operatorname{ps}}
\newcommand{\adm}{\operatorname{adm}}

\renewcommand{\div}{\operatorname{div}}

\newcommand{\car}{\operatorname{char}}

\newcommand{\CC}{\operatorname{\ov{\m C}}}

\newcommand{\M}{\operatorname{\ov{\m M}}}

\newcommand{\MM}{\operatorname{\ov{M}}}

\newcommand{\NE}{\operatorname{NE}}
\newcommand{\NEb}{\operatorname{\ov{NE}}}

\newcommand{\Gm}{\operatorname{\mathbb{G}_m}}

\usepackage{tikz}
\usetikzlibrary{decorations.pathmorphing}
\tikzset{dot/.style={
		circle,
		fill=black,
		inner sep=1.5pt,
}}

\title{On  some modular contractions of the moduli space of stable pointed curves}

\author{Giulio Codogni}
\address{Dipartimento di Matematica, Universit\`a degli Studi di Roma Tor Vergata, Via della ricerca scientifica, 00133 Roma, Italy.}
\email{codogni@mat.uniroma2.it}

\author{Luca Tasin}
\address{Dipartimento di Matematica F.\ Enriques, Universit\`a degli Studi di Milano, Via Cesare Saldini 50, 20133 Milano, Italy} 
\email{luca.tasin@unimi.it}

\author{Filippo Viviani}
\address{Dipartimento di Matematica e Fisica, Universit\`{a} degli Studi Roma Tre,  Largo San Leonardo Murialdo, 00146 Rome, Italy} 
\email{viviani@mat.uniroma3.it}


\begin{document}

\maketitle

\begin{abstract}
	The aim of this paper is to study  some modular contractions of the moduli space of stable pointed curves  $\MM_{g,n}$.  These new moduli spaces, which are modular compactifications  of $\mathrm M_{g,n}$,  are related to the minimal model program for $\MM_{g,n}$ and have been introduced in  \cite{CTV1}. We  interpret them  as log canonical models of adjoint  divisors and  we then  describe the Shokurov decomposition of a region of boundary divisors on $\MM_{g,n}$.
\end{abstract}



\section{Introduction}

The moduli space $\MM_{g,n}$ of stable $n$-pointed  curves of genus $g$ is a natural compactification of the moduli space $M_{g,n}$ of smooth $n$-pointed projective curves of genus $g$ and it is one of the most studied objects in algebraic geometry. Nevertheless, most of its rich birational geometry is still unknown.  In particular, the following two natural questions are still very much open.

\begin{question}[Mumford] \label{Q:1} 
What is the nef cone of $\MM_{g,n}$?
\end{question}

The F-conjecture, usually attributed to Fulton, predicts that a divisor $L$ is nef if and only if it has non-negative intersection with the F-curves
($L$ is called F-nef if it intersects non-negatively all the F-curves), which are the $1$-dimensional strata of the  stratification of $\MM_{g,n}$ by dual graphs; see Section  \ref{S:prel} for details. This would have the striking consequence that the nef cone of $\MM_{g,n}$ is rational polyhedral. In the breakthrough paper \cite{GKM}, Gibney-Keel-Morrison reduce the F-conjecture to the genus $0$ case, which however remains still widely open. In the same paper, the authors pose the following 

\begin{question}[Gibney-Keel-Morrison \cite{GKM}] \label{Q:2}  
What are all the contractions (i.e. separable morphisms with connected fibres to projective varieties) of $\MM_{g,n}$?
\end{question}

Note that the contractions of $\MM_{g,n}$ correspond to the faces of the semiample cone of $\MM_{g,n}$ which is a subcone of the nef cone of $\MM_{g,n}$ (and indeed a proper subcone at least if $g \ge3$, $n>0$ and the characteristic is zero by \cite{Kee}). 
In the  paper \cite{GKM}, the authors prove that any contraction $\MM_{g,n}\to X$ factors through a forgetful morphism $\MM_{g,n}\to \MM_{g,m}$ for some $m\leq n$ followed by a birational contraction $\MM_{g,m}\to X$ that is an isomorphism in the interior $M_{g,m}$. In particular, any birational contraction $\MM_{g,n}\to X$ is an isomorphism on $M_{g,n}$, so that $X$ is a new compactification of $M_{g,n}$. 

In our previous paper \cite{CTV1}, we introduced several new birational contractions $\Upsilon_T:\MM_{g,n}\to \MM_{g,n}^T$, whose codomains $\MM_{g,n}^T$ 
are  weakly modular compactifications of $M_{g,n}$ in the sense of \cite[Sec. 2.1]{FS}. 
Since the number of these birational contractions grows exponentially in $(g,n)$ (see Remark \ref{R:number} for the exact count), this significantly expands the known examples of birational contractions of $\MM_{g,n}$, that previous to \cite{CTV1}, to the best of our knowledge, consisted of the first two steps of the Hassett-Keel program (see \cite{HH13}, \cite{AFSV1, AFS2, AFS3}) and, for $n=0$, the Torelli morphism from $\MM_g$ to the Satake compactification of the moduli space of principally polarized abelian varieties. 

The aim of this paper, which is a sequel of \cite{CTV1}, is to study the geometry of the variety $\MM_{g,n}^T$ and of the birational contraction $\Upsilon_T:\MM_{g,n}\to \MM_{g,n}^T$. 
Moreover, we describe $\Upsilon_T:\MM_{g,n}\to \MM_{g,n}^T$ as the ample model of suitable adjoint divisors on $\MM_{g,n}$ (see Theorem \ref{T:main}).  
With an \emph{adjoint divisor} on $\MM_{g,n}$ we mean a $\mathbb Q$-divisor $L$ of the form $K_{\M_{g,n}} + \psi +a\lambda + \Delta$, where $\psi$ is the total cotangent bundle, $\lambda$ is the Hodge line bundle, $a$ is a non-negative rational number and $\Delta$ is an effective boundary divisor with  coefficients at most one (see Definition \ref{D:adj}, and Remark \ref{R:poladj} for motivations).

As a consequence of Theorem \ref{T:main}, we are then able to  describe the decomposition in Shokurov polytopes of a region of the polytope of adjoint divisors (see Corollary \ref{C:main}). Recall that a Shokurov polytope collects adjoint  divisors with the same ample model; the existence of such a decomposition for $\MM_{g,n}$  is proven in \cite[Cor. 1.1.5]{BCHM10}. 
Determining the full decomposition of the space of adjoint divisors is one of the ultimate goals in the study of the birational geometry of $\MM_{g,n}$ and this is the first 
general result in this direction. Let us stress again that the ample models of the Shokurov polytopes described by our result all have a modular interpretation, so it is natural to ask the following question.

\begin{question}
	Does the Shokurov decomposition of the space of adjoint divisors  on $\MM_{g,n}$ admit a modular interpretation?
\end{question}

Let us also mention that in \cite{CTV1} we also constructed several other weakly modular compactifications $\MM_{g,n}^{T+}$ of $M_{g,n}$, which are endowed with a morphism 
$f_T^+:\MM_{g,n}^{T+}\to \MM_{g,n}^T$ that is a  flip (with respect to a suitable divisor) of $\Upsilon_T:\MM_{g,n}\to \MM_{g,n}^T$. However, the varieties $\MM_{g,n}^{T+}$  are only rational (and not regular) contractions of  $\MM_{g,n}$ and we will not study them in this paper.

\subsection{Description of the results}

In order to explain more in details the results of this paper, let us recall the definition of the  varieties $\MM_{g,n}^T$. 
Consider the set of indexes 
\begin{equation}\label{E:setT}
T_{g,n}:=\{\irr \} \cup \{ (\tau,I) : 0\leq \tau\leq g, I \subseteq [n]:=\{1,\ldots, n\}, (\tau,I)\neq (0,\emptyset), (g,[n]) \}/\sim.
\end{equation}
where $\sim$ is the equivalence relation such that $\irr$ is equivalent only to itself and $(\tau,I) \sim (\tau',I')$ if and only if $(\tau,I)= (\tau',I')$ or $(\tau',I')=(g-\tau,I^c )$, where $I^c=[n] \setminus I$. We will denote the class of $(\tau,I)$ in $T_{g,n}$ by $[\tau, I]$ and the class of $\irr$ in $T_{g,n}$ again by $\irr$. We also set $T_{g,n}^*:=T_{g,n}\setminus \{\irr\}$. 

For any $T \subset T_{g,n}$,  consider the  (smooth, irreducible and of finite type over the base field $k$) algebraic stack of $T$-semistable curves 
\begin{center}
$\M_{g,n}^T := \{$$n$-pointed curves of genus $g$ with ample log canonical class, having singularities that are nodes, cusps, or tacnodes of type contained in $T$, and not having elliptic tails$\}$, 
\end{center}
see Definitions \ref{D:type} and \ref{D:pseudostable} for details. Note that there are open inclusions  among the different stacks $\M_{g,n}^T$: if $T\subseteq S$ then $\M_{g,n}^T\subseteq \M_{g,n}^S$ and in Proposition \ref{P:equaT} we study when the converse is true.

As special case of the above stacks, for $T=\emptyset$  we obtain the stack of pseudostable $n$-pointed curves of genus $g$ 
\begin{center}
$\M_{g,n}^{\ps} :=\M_{g,n}^{\emptyset}=\{$$n$-pointed curves of genus $g$ with ample log canonical class, having singularities that are nodes and cusps, and not having  elliptic tails$\}$.
\end{center}

Excluding the trivial case $(g,n)=(1,1)$ (when $\M_{g,n}^{\ps}=\emptyset$) and the pathological case $(g,n)= (2,0)$ (see \cite[Rmk. 1.13]{CTV1} and also Remark \ref{R:Tclosed-g2}), $\M_{g,n}^{\ps}$ is  a proper Artin stack with finite inertia (and Deligne-Mumford if $\car(k)\neq 2,3$) with coarse moduli space $\phi^{\ps}:\M_{g,n}^{\ps} \to \MM_{g,n}^{\ps}$ which is a normal projective variety. 
Moreover, there is a regular birational morphism $\wh \Upsilon:\M_{g,n}\to \M_{g,n}^{\ps}$ which sends a stable curve into the pseudostable curve obtained by replacing each elliptic tail with a cusp. The induced morphism $\Upsilon:\MM_{g,n}\to \MM_{g,n}^{\ps}$ on coarse moduli spaces is the contraction of the $K$-negative extremal ray of the Mori cone of $\MM_{g,n}$ generated by the elliptic tail curve $C_{\rm ell}\subset \MM_{g,n}$, which  parametrises stable curves obtained by attaching a fixed smooth curve in $M_{g-1,n+1}$ with a moving elliptic tail. For a proof of these results, see  \cite{Sch} and  \cite{HH1} for $n=0$, and \cite[Prop. 1.11 and Sec. 3]{CTV1} for general $n$.

Returning to the stacks $\M_{g,n}^T$ with $T$ arbitrary, we proved in \cite{CTV1} (see also Theorem \ref{T:goodsp}) that, if $\car(k)$ is big enough with respect to $(g,n)$ and we exclude the pathological case $(g,n)=(2,0)$ (see Remarks \ref{R:Tclosed-g2}, \ref{R:Tstable-g2}  and \ref{R:mod-g2}),
 then the stack $\M_{g,n}^T$ admits a good moduli space  $\MM_{g,n}^T$ which is a normal and irreducible proper algebraic space.  Moreover, there is a birational contraction $\Upsilon_T:  \MM_{g,n} \to \MM_{g,n}^T$, which factorises as 
$$
\Upsilon_T=f_T \circ \Upsilon : \MM_{g,n} \xrightarrow{\Upsilon} \MM_{g,n}^{\ps} \xrightarrow{f_T} \MM_{g,n}^T.
$$ 
Furthermore, if ${\rm char}(k)=0$, then $f_T: \MM_{g,n}^{\ps} \to \MM_{g,n}^T$ is the contraction associated to a $K$-negative face $F_T$ of the Mori cone $\overline{NE}(\MM_{g,n}^{\ps})$ and $\MM_{g,n}^T$ is a projective variety. 
When $T=T_{g,n}$, $f_T$ is the second step of the Hassett-Keel program, and has been studied (together with its flip) in \cite{HH13}  for $n=0$ and  in the trilogy \cite{AFSV1, AFS2, AFS3} for $n>0$ (and $\car(k)=0$).

In Section \ref{Sec:Tss}, we study the geometric  properties of $\MM_{g,n}^T$ and of the birational contraction $f_T:\MM_{g,n}^{\ps}\to \MM_{g,n}^T$. 
More precisely: in Proposition \ref{P:descend} we determine which line bundles on $\M_{g,n}^T$ descend to $\bbQ$-line bundles on $\MM_{g,n}^T$; in Proposition \ref{P:QfatQG}, we investigate when $\MM_{g,n}^T$ is $\bbQ$-factorial or $\bbQ$-Gorenstein; in Proposition \ref{P:factor}, we show that the contraction $f_T$ can be factorised in a modular way into a composition of divisorial contractions followed by a small contraction.

Finally, Section \ref{Sec:ampl-mod} is devoted to the description of the contractions $\Upsilon_T:\MM_{g,n}\to \MM_{g,n}^T$ as ample models of adjoint $\bbQ$-divisors on $\MM_{g,n}$ (in characteristic zero). The main result of this paper is the following theorem. 

\begin{theorem}[Theorem \ref{T:all-one}] \label{T:main}
	Assume that ${\rm char}(k)=0$ and let $L$ be a $\bbQ$-divisor on $\MM_{g,n}$ of the form
	$$
	L=K+\psi+a\lambda+ \alpha_{\irr}\delta_{\irr} + \sum_{[i,I]\in T^*_{g,n}}\alpha_{i,I}\delta_{i,I}
	$$
	where $a \ge 0$, $0 \le \alpha_{\irr} \le 1$ and  $0 \le \alpha_{i,I} \le 1$. 
	
	\begin{enumerate}
		\item $L$ is ample if and only if it is F-ample. In this case, we have that  
		$$ \MM_{g,n}= \Proj \bigoplus_{m \ge 0} H^0(\MM_{g,n}, \lfloor mL \rfloor).$$
		\item  Assume that $(g,n)\neq (1,1), (2,0)$. Then $L$ is semiample with associated contraction equal to $\Upsilon:\MM_{g,n}\to \MM_{g,n}^{\ps}$
		if and only if it is F-nef and the only F-curve on which it is trivial is $C_{\rm ell}$. In this case, we have that 
		$$ \MM_{g,n}^{\ps}= \Proj \bigoplus_{m \ge 0} H^0(\MM_{g,n}, \lfloor mL \rfloor).$$
		\item Fix $T\subseteq T_{g,n}$ and assume that $(g,n)\neq (1,1), (2,0), (1,2)$ and that $\alpha_{\irr}\leq \frac{10-a}{12}$.  Then 
		$\Upsilon_*(L)$ is semiample with associated contraction equal to $f_T:\MM_{g,n}^{\ps}\to \MM_{g,n}^T$ if and only if  $\Upsilon^*(\Upsilon_*(L))$ is F-nef and the only F-curves on which it is trivial are the ones whose images in $\MM_{g,n}^{\ps}$ have numerical classes contained in $F_T$. Moreover, in this case the ample model of $L$ is equal to 
		$\Upsilon_T:\MM_{g,n}\to \MM_{g,n}^T$ if we assume furthermore that 
		$\displaystyle \alpha_{\irr}\leq \frac{9-a+\alpha_{1,\emptyset}}{12}$.
		
		In particular, we have that 
		$$ \MM_{g,n}^{T}= \Proj \bigoplus_{m \ge 0} H^0(\MM_{g,n}, \lfloor mL \rfloor)=\Proj \bigoplus_{m \ge 0} H^0(\MM_{g,n}^{\ps}, \lfloor m\Upsilon_*(L) \rfloor).$$
	\end{enumerate}
\end{theorem}

The proof of the above theorem is based on two key observations.  
The first one (Proposition \ref{P:F-conjecture}) is that an adjoint divisor $L$ is nef if and only if it is F-nef, i.e.\ adjoint divisors satisfy the F-conjecture, which is an interesting statement by itself.
The second result (Proposition \ref{P:1strata}) says that  if $L$ is an adjoint divisor on $\MM_{g,n}^{\ps}$ such that $\Upsilon^*(L)$ is F-nef and the only F-curves on which it is trivial are the ones whose images in $\MM_{g,n}^{\ps}$ have numerical classes contained in $F_T$, then $L$ is nef and trivial only on the curves whose numerical class is contained in $F_T$.  

The following Corollary of Theorem \ref{T:main} describes the decomposition in Shokurov polytopes of a region of adjoint divisors. As explained in Remark \ref{R:poladj}, our divisors are adjoint in the generalised sense of \cite{Birkar-Zhang}.

\begin{cor}[Corollary \ref{C:dec-poly}]\label{C:main}
	Assume that ${\rm char}(k) =0$ and that $(g,n)\neq (1,1), (2,0), (1,2)$. 
	Let $L$ be an adjoint $\bbQ$-divisor on $\MM_{g,n}$
	$$
	L=K+\psi+a\lambda+ \alpha_{\irr}\delta_{\irr} + \sum_{[i,I]\in T^*_{g,n}}\alpha_{i,I}\delta_{i,I},
	$$
	 $a \ge 0$, $0 \le \alpha_{\irr} \le 1$ and  $0 \le \alpha_{i,I} \le 1$ such that $\vert \alpha_{i,I} -\alpha_{j,J}\vert <\frac{1}{3}$ for any $[i,I], [j,J] \in T_{g,n}^*$ and such that if $\alpha_{\irr}=1$ then $\alpha_{i,I}>0$ for any $[i,I]\in T_{g,n}^*$. 
	Assume furthermore that  
	\begin{equation}\label{E:Iinq1}
	\frac{7-a}{10}\leq \alpha_{\irr} \quad (\text{for }g\geq 2), 
	\end{equation}
	\begin{equation}\label{E:Iinq2}
	\frac{7-a+\alpha_{i,I}+\alpha_{i+1,I}}{12}\leq \alpha_{\irr} \text{ for any }[i,I], [i+1,I]\in T^*_{g,n}\setminus \{[1,\emptyset]\}.
	\end{equation}
	Then the ample model of $L$ is 
	\begin{itemize}
		\item $\id:\MM_{g,n}\to \MM_{g,n}$ if $\frac{9-a+\alpha_{1,\emptyset}}{12}<\alpha_{\irr}$;
		\item $\Upsilon: \MM_{g,n}\to \MM_{g,n}^{\ps}$ if $\frac{9-a}{12}<\alpha_{\irr}\leq \frac{9-a+\alpha_{1,\emptyset}}{12}$;
		\item $\Upsilon_T:\MM_{g,n}\to \MM_{g,n}^T$ if $\alpha_{\irr}\leq \frac{9-a}{12} $  where $T$ is admissible (see Definition \ref{def:admissible}) and it is uniquely determined by  
		$$
		\begin{aligned} 
		& (\text{for }g\geq 2) \quad \irr \in T \Leftrightarrow \text{ equality holds in }\eqref{E:Iinq1},\\
		& \{[i,I],[i+1,I]\}\subseteq T \Leftrightarrow \text{ equality holds in }\eqref{E:Iinq2}.
		\end{aligned}
		$$
	\end{itemize}
	
\end{cor}

In particular, this corollary describes the Shokurov decomposition for the region of adjoint divisors 
$$
L=K+\psi+a\lambda+ \alpha_{\irr}\delta_{\irr} + \sum_{[i,I]\in T^*_{g,n}}\alpha_{i,I}\delta_{i,I}
$$
such that $\frac{9-a}{12} \le \alpha_{\irr} \le 1$ and $\frac{2}{3} < \alpha_{i,I} \le 1$ for any $[i,I]\in T_{g,n}^*$.

\subsection*{Acknowledgment}
We had the pleasure and the benefit of conversations with  J.\ Alper, E.\ Arbarello, G.\ Farkas, M.\ Fedorchuk, R.\ Fringuelli,  A.\ Lopez,  Zs.\ Patakfalvi and R.\ Svaldi about the topics of this paper.

The first author is supported by prize funds related to the PRIN-project 2015 EYPTSB ``Geometry of Algebraic Varieties" and by University Roma Tre. The second author was supported during part of this project by the DFG grant ``Birational Methods in Topology and Hyperk\"ahler Geometry". The third author is  a member of the CMUC (University of Coimbra), where part of this work was carried over. 

The authors are members of the GNSAGA group of INdAM.

\section{Preliminaries: the Picard groups and some special curves  of $\MM_{g,n}$ and of $\MM_{g,n}^{\ps}$}\label{S:prel}

The aim of this section is to recall the description of the rational Picard groups of $\M_{g,n}$ and of $\M_{g,n}^{\ps}$, and of their coarse moduli spaces $\MM_{g,n}$ and $\MM_{g,n}^{\ps}$, and to introduce some special curves in 
$\MM_{g,n}$ and $\MM_{g,n}^{\ps}$ that will play a key role in the sequel. 




Let us start by defining the tautological  line bundles on $\M_{g,n}$. 
To any element of the set $T_{g,n}$ defined in \eqref{E:setT}, we can associate a line bundle on $\M_{g,n}$
in the following way:
\begin{itemize}
\item To $\irr\in T_{g,n}$ we associate the boundary line  bundle $\delta_{\irr}:=\cO_{\M_{g,n}}(\Delta_{\irr})$, where $\Delta_{\irr}$ is the irreducible boundary divisor  of $\M_{g,n}$ whose generic point is a stable curve with one non separating node.
\item To $[i,I]\in T_{g,n}^*$, which is different from any subset of the form $[0,\{k\}]\in T_{g,n}$ for $k\in [n]$, we associate the boundary line  bundle $\delta_{i,I}:=\cO_{\M_{g,n}}(\Delta_{i,I})$, where $\Delta_{i,I}$ is the irreducible boundary divisor  of $\M_{g,n}$ whose generic point is a stable curve formed by  two smooth irreducible curves $C_1\in \M_{i,I\cup\{\star\}}$ and $C_2\in \M_{g-i,I^c\cup\{\bullet\}}$ glued nodally by identifying $\star$ and $\bullet$. 

\item To $[0,\{k\}]\in T_{g,n}^*$, we associate the $k$-th cotangent line bundle $\psi_k:=\sigma_k^*(\omega_{\CC_{g,n}/\M_{g,n}})$, where $\omega_{\CC_{g,n}/\M_{g,n}}$ is the relative dualising sheaf of the universal family $\pi:\CC_{g,n}\to \M_{g,n}$ and $\sigma_k$ is its $k$-th section. 

\end{itemize}
Following \cite{GKM}, we will set $\delta_{0,\{i\}}=-\psi_i$ so that the  line bundles $\delta_{i,I}$ are defined for every $[i,I]\in T_{g,n}^*$. As customary, we will denote the total cotangent and total boundary line bundles by (using additive notation)
\begin{equation}\label{E:deltapsi}
\begin{aligned}
& \psi:=\sum_{i=1}^n \psi_i, \\
& \delta:=\delta_{\irr}+\sum\delta_{i,I}, \\
\end{aligned}
\end{equation}
where the last sum ranges over all elements $[i,I]\in T_{g,n}^*$ such that 
$[i,I]\neq [0,\{k\}]$ for some $1\leq k \leq n$. Recall also that on $\M_{g,n}$ we can define the Hodge line bundle $\lambda:=\det \pi_*(\omega_{\CC_{g,n}/\M_{g,n}})$.

\begin{fact}\label{F:PicM}
The rational Picard group $ \Pic(\M_{g,n})_{\bbQ}$ is generated by the tautological line bundles $\lambda$, $\delta_{\irr}$ and $\{\delta_{i,I}\}_{[i,I]\in T_{g,n}^*}$. 
\end{fact}   
The above fact is proven in \cite[Thm. 2.2]{AC}  for ${\rm char}(k)=0$ and in \cite{Mor} for ${\rm char}(k)>0$. 
Moreover, if $g\geq 3$ there are no relations among the tautological line bundles while for $g=0,1,2$ the list of all relations can be found in loc. cit. Quite recently, Fringuelli and the third author \cite{FV} proved that, if $\car(k)\neq 2$, then the integral Picard group $\Pic(\M_{g,n})$ is generated by the tautological line bundles (generalising  the result of \cite{AC0} for $\car(k)=0$ and $g\geq 3$).

Since the rational Picard group of the coarse moduli space $\MM_{g,n}$ can be identified with  the rational Picard group  of the  stack $\M_{g,n}$ via the pull-back along the morphism $\phi: \M_{g,n}\to \MM_{g,n}$, and since $\MM_{g,n}$ has 
finite quotient singularities, and hence it is $\bbQ$-factorial, we deduce the following

\begin{cor}\label{F:PicMM}
The group $\Pic(\MM_{g,n})_{\bbQ}= \Cl(\MM_{g,n})_{\bbQ}$ is generated by $\lambda$, $\delta_{\irr}$ and $\{\delta_{i,I}\}_{[i,I]\in T_{g,n}^*}$.
\end{cor}

A special class of curves of $\MM_{g,n}$ that will play a crucial role in the sequel are the  \emph{F-curves}, which are the one-dimensional strata of the stratification of  $\MM_{g,n}$ by dual graphs.
Up to numerical equivalence, the $F$-curves can be described in the following way  (see \cite[Thm. 2.2]{GKM}):

\begin{enumerate}[(1)]
\item For $g\geq 1$: $C_{\rm ell}$ (called the elliptic tail curve) is obtained by attaching a fixed curve of $\MM_{g-1,[n]\cup\{n+1\}}$ to a moving curve in $\MM_{1,\{n+1\}}$ (and stabilising if necessary).
\item For $g\geq 3$: $F(\irr)$ is obtained by attaching a fixed curve of $\MM_{g-3,[n]\cup\{n+1,n+2,n+3,n+4\}}$ to a moving curve in $\MM_{0,\{n+1,n+2,n+3,n+4\}}$.
\item For $0\leq i \leq g-2$ and $I\subseteq [n]$ such that $(i,I)\neq (0,\emptyset)$:    $F([i,I])$ is obtained by attaching a fixed curve of $\MM_{i,I\cup \{n+1\}}$ and a fixed curve of $\MM_{g-2-i,I^c\cup \{n+2,n+3,n+4\}}$ to a moving curve in $\MM_{0,\{n+1,n+2,n+3,n+4\}}$ (and stabilising if necessary). 
\item For $1\leq i \leq g-1$: $F_s([i,I])$ is obtained by attaching a fixed curve of $\MM_{i-1,I\cup \{n+1,n+2\}}$ and a fixed curve of $\MM_{g-1-i,I^c\cup \{n+3,n+4\}}$ to a moving curve in $\MM_{0,\{n+1,n+2,n+3,n+4\}}$ (and stabilising if necessary). 
\item For $0\leq i, j$  with $i+j\leq g-1$ and  disjoint subsets $I, J\subseteq [n]$ such that $(i,I), (j,J)\neq (0,\emptyset)$: 
$F([i,I], [j,J])$ is obtained by attaching a fixed curve of $\MM_{i,I\cup \{n+1\}}$, a fixed curve of  $\MM_{j,J\cup \{n+2\}}$ and a fixed curve of $\MM_{g-1-i-j,(I\cup J)^c\cup \{n+3,n+4\}}$ to a moving curve in $\MM_{0,\{n+1,n+2,n+3,n+4\}}$ (and stabilising if necessary).
\item For $0\leq i, j, k$  with $i+j+k \leq g$ and  pairwise disjoint subsets $I, J, K\subseteq [n]$ such that  $(i,I), (j,J), (k,K), (g-i-j-k,(I\cup J\cup K)^c) \neq (0,\emptyset)$: $F([i,I], [j,J], [k,K])$ is obtained by attaching a fixed curve of $\MM_{i,I\cup \{n+1\}}$, a fixed curve of  $\MM_{j,J\cup\{n+2\}}$, a fixed curve  of  $\MM_{k,K\cup \{n+3\}}$ and a fixed curve of $\MM_{g-i-j-k,(I\cup J\cup K)^c\cup \{n+1\}}$ to a moving curve in $\MM_{0,\{n+1,n+2,n+3,n+4\}}$ (and stabilising if necessary).
\end{enumerate}

The intersections of the $\bbQ$-line bundles of $\MM_{g,n}$ with the F-curves are determined by the following formulae.

\begin{lemma}(\cite[Thm. 2.1]{GKM})\label{L:forGKM}
Given a $\bbQ$-line bundle $L=a\lambda- b_{\irr}\delta_{\irr}-\sum_{[i,I]\in T^*_{g,n}} b_{i,I}\delta_{i,I}$ on $\MM_{g,n}$, the intersection of $L$ with the  F-curves is given by
\begin{enumerate}[(1)]
\item $L.C_{\rm ell}=a-12b_{\irr}+b_{1,\emptyset}$,
\item $L.F(\irr)=b_{\irr}$, 
\item $L.F([i,I])=b_{i,I}$, 
\item  $L.F_s([i,I])=2b_{\irr}-b_{i,I}$, 
\item $L.F([i,I], [j,J])=b_{i,I}+b_{j,J}-b_{i+j,I\cup J}$, 
\item  $L.F([i,I], [j,J], [k,K])=b_{i,I}+b_{j,J}+b_{k,K}-b_{i+j,I\cup J}-b_{i+k,I\cup K}-b_{j+k,J\cup K}+b_{i+j+k,I\cup J\cup K}$.  
\end{enumerate} 
\end{lemma}

We will now   recall the description of the rational Picard group of the stack $\M_{g,n}^{\ps}$ and of its coarse moduli space $\MM_{g,n}^{\ps}$.
Note that the tautological  line bundles $\{\lambda,\delta_{\irr}, \{\delta_{i,I}\}_{[i,I]\in T_{g,n}^*}\}$ can be first restricted to $\M_{g,n}\setminus \Delta_{1,\emptyset}$ and then uniquely extended to a line bundle on $\M_{g,n}^{\ps}$ using that 
$\M_{g,n}^{\ps}$ is smooth and that $\M_{g,n}\setminus \Delta_{1,\emptyset}$ is  an open subset of $\M_{g,n}^{\ps}$ whose complement has codimension greater or equal to two. 
We will denote the line bundles on $\M_{g,n}^{\ps}$ obtained in this way by $\{\lambda^{\ps},\delta_{\irr}^{\ps}, \{\delta_{i,I}^{\ps}\}_{[i,I]\in T_{g,n}^*}\}$, or simply, with a slight abuse of notation, as $\{\lambda,\delta_{\irr}, 
\{\delta_{i,I}\}_{[i,I]\in T_{g,n}^*}\}$. Note that $\delta_{1,\emptyset}^{\ps}=0$ by construction.

\begin{proposition}\label{P:PicMps}
Assume that $(g,n)\neq (2,0), (1,1)$. 
\begin{enumerate}[(i)]
\item \label{P:PicMps1} The group $\Pic(\M_{g,n}^{\ps})_{\bbQ}$ is generated by the tautological line bundles $\lambda$, $\delta_{\irr}$ and $\{\delta_{i,I}\}_{[i,I]\in T_{g,n}^*\setminus \{[1,\emptyset]\}}$.
\item \label{P:PicMps2} Assume that $\car(k)\neq 2,3$. The natural inclusion  $\Pic(\MM_{g,n}^{\ps})_{\bbQ}\hookrightarrow \Cl(\MM_{g,n}^{\ps})_{\bbQ}$ is an equality and the pull-back $(\phi^{\ps}):\Pic(\MM_{g,n}^{\ps})_{\bbQ}\to \Pic(\M_{g,n}^{\ps})_{\bbQ}$ is an isomorphism.
\item \label{P:PicMps3} The push-foward map $\Upsilon_*:\Cl(\MM_{g,n})_{\bbQ}\to \Cl(\MM_{g,n}^{\ps})_{\bbQ}$ is determined by
$$\begin{sis}
& \Upsilon_*(\lambda)=\lambda, \\
& \Upsilon_*(\delta_{\irr})=\delta_{\irr}, \\
& \Upsilon_*(\delta_{i,I})=
\begin{cases}
\delta_{i,I} \quad & \text{ for any } [i,I]\neq [1,\emptyset], \\
0 \quad & \text{ for  } [i,I]=[1,\emptyset],\\
\end{cases}
\end{sis}$$ 
while, if $\car(k)\neq 2,3$, the pull-back map $\Upsilon^*:\Pic(\MM_{g,n}^{\ps})_{\bbQ}\to \Pic(\MM_{g,n})_{\bbQ}$ is determined by 
$$\begin{sis}
& \Upsilon^*(\lambda)=\lambda+\delta_{1,\emptyset}, \\
& \Upsilon^*(\delta_{\irr})=\delta_{\irr}+12\delta_{1,\emptyset}, \\
& \Upsilon^*(\delta_{i,I})=\delta_{i,I} \quad \text{ for any } [i,I]\neq [1,\emptyset]. 
\end{sis}$$ 
\end{enumerate}
\end{proposition}
\begin{proof}
Part \eqref{P:PicMps1} follows from the fact that the restriction map $\Pic(\M_{g,n})\twoheadrightarrow \Pic(\M_{g,n}\setminus \Delta_{1,\emptyset})$ is surjective because $\M_{g,n}$ is smooth while the restriction map  $\Pic(\M_{g,n})\xrightarrow{\cong} \Pic(\M_{g,n}\setminus \Delta_{1,\emptyset})$ is an isomorphism since $\M_{g,n}^{\ps}$ is smooth and that $\M_{g,n}\setminus \Delta_{1,\emptyset}$ is  an open subset of $\M_{g,n}^{\ps}$ whose complement has codimension greater or equal to two (see also \cite[Cor. 1.29]{CTV1}).  

Part \eqref{P:PicMps2}  follows from the fact that $\phi^{\ps}:\M_{g,n}^{\ps}\to \MM_{g,n}^{\ps}$ is a coarse moduli space and that, if $\car(k)\neq 2,3$, $\MM_{g,n}$ has  finite quotient singularities, and hence it is $\bbQ$-factorial (see \cite[Prop. 3.1(i)]{CTV1}).

Part \eqref{P:PicMps3}: the formulas for $\Upsilon_*$ are obvious from the definition of the generators of $\Cl(\MM_{g,n})_{\bbQ}$ and of $\Cl(\MM_{g,n}^{\ps})_{\bbQ}$; the formulas for $\Upsilon^*$ are proved in \cite[Prop. 3.5(iii)]{CTV1}.
\end{proof}

We now introduce some special curves in $\MM_{g,n}^{\ps}$, that will play a key role in the sequel. 


\begin{definition}(\cite[Def. 0.1]{CTV1})\label{D:ell-bridg}
The \emph{elliptic bridge curves} are the following curves of $\MM_{g,n}^{\ps}$, well-defined up to numerical equivalence: 
\begin{itemize}
\item $C(\irr):=\Upsilon(F_s([1,\emptyset]))$ is the elliptic bridge curve of type $\{\irr\}$, if $g\geq 2$;
\item $C([\tau,I],[\tau+1,I]):=\Upsilon(F([\tau,I],[g-\tau-1,I^c]))$ is the elliptic bridge curve of type $\{[\tau,I],[\tau+1,I]\}$, for any $\{[\tau,I],[\tau+1,I]\}\subseteq T_{g,n}^*- \{[1,\emptyset]\}$.
\end{itemize}
\end{definition}

The intersections of the $\bbQ$-line bundles of $\MM_{g,n}^{\ps}$ with the elliptic bridge curves are determined by the following formulae.

\begin{lemma}\label{L:int-C}(\cite[Lemma 3.8]{CTV1})
Assume that $\car(k)\neq 2,3$. Given a $\bbQ$-line bundle $L=a\lambda+b_{\irr}\delta_{\irr}+\sum_{[i,I]\in T^*_{g,n}-\{[1,\emptyset]\}} b_{i,I} \delta_{i,I} $ in $\M_{g,n}^{\ps}$, we have the following intersection formulas
$$\begin{sis}
& C(\irr) \cdot L= a+10 b_{\irr}, \\
& C([\tau, I], [\tau+1, I])\cdot L=a+12 b_{\irr}-b_{\tau,I}-b_{\tau+1,I}.
\end{sis}$$
\end{lemma}

\section{The stack of $T$-semistable curves}\label{Sec:Tss}

The aim of this subsection is to study the  stack of $T$-semistable curves, whose definition we now recall (following the terminology of \cite[Sec. 1]{CTV1}).

\begin{definition}\label{D:type}[Types of tacnodes](\cite[Def. 1.6]{CTV1})
Let $(C, \{p_i \}_{i=1}^n)$ be a $n$-pointed curve such that $C$ is Gorenstein and $\omega_C(\sum_{i=1}^n p_i)$ is ample.
Let $p \in C$ be a tacnode. We say that $p$ is of type:
\begin{itemize}
\item $\type(p):=\{\irr\}\subseteq T_{g,n} $ if the normalisation of $C$ at  $p$ is connected;
\item  $\type(p):=\{[\tau,I],[\tau+1,I]\} \subseteq T_{g,n}$ if the normalisation of $C$ at $p$ consists of two connected components, one of which has arithmetic genus  $\tau$ and marked points $\{p_i\}_{i \in I}$ (and then the other one will have arithmetic genus $g-\tau-1$ and marked points $\{p_i\}_{i\in I^c}$).
\end{itemize}
\end{definition}

\begin{definition}\label{D:pseudostable} (\cite[Def. 1.16]{CTV1})
Let $T\subseteq T_{g,n}$.
\begin{enumerate}[(i)]
\item A \emph{$T$-semistable}  ($n$-pointed) curve  is an $n$-pointed curve $(C, \{p_i \}_{i=1}^n)$ such that:
\begin{enumerate}[(a)]
\item $C$ has only nodes, cusps, and tacnodes of type contained in $T$ as singularities;
\item $C$ does not have $A_1$-attached and $A_3$-attached elliptic tails;
\item $\omega_{C}(\sum p_i)$ is ample.
\end{enumerate}
\item The stack of $T$-semistable $n$-pointed curves of genus $g$, denoted by $\M_{g,n}^{T}$, parametrises flat, proper families of $n$-pointed curves $(\pi: \m C \to B,\{\sigma_i\}_{i=1}^n)$, where $\{\sigma_i\}_{i=1}^n$ are distinct sections that lie in the smooth locus of $\pi$, such that the line bundle $\omega_{\m C /B}(\sum \sigma_i)$ is relatively ample and the geometric fibres of $\pi$ are  $T$-semistable  $n$-pointed curves of  genus $g$.
\end{enumerate}
\end{definition}
Recall that the stack $\M_{g,n}^{T}$ is a smooth and  irreducible algebraic stack of finite type over $k$ (see \cite[Theorem 1.19]{CTV1}). From the definition, it is easy to see that $\M_{g,n}^{\emptyset}=\M_{g,n}^{\ps}$. 

Let us study the relations among the stacks $\M_{g,n}^T$. Note that if $T\subseteq T'$ then $\M_{g,n}^T\subseteq \M_{g,n}^{T'}$ but it may very well be the case that $\M_{g,n}^T=\M_{g,n}^{T'}$ for $T\subsetneq T'$. In order to characterise when this happens, we introduce the following notions.

\begin{definition}\label{def:admissible}
\noindent
\begin{enumerate}[(i)]
\item A subset  $T \subseteq T_{g,n}$  is called \emph{admissible} if $[1,\emptyset]\not\in T$ and for every $[\tau,I]$ in $T$ then either $[\tau-1,I]$ or $[\tau+1,I]$ are in $T$. If $g=1$, we also require that $\irr \not \in T$.
\item Given a subset $T\subset T_{g,n}$, we obtain an admissible subset $T^{\adm}\subseteq T$ in the following two steps:
\begin{itemize}
\item first we set $\wt T:=T-\{[1,\emptyset]\}$ if $g\geq 2$ and $\wt T:=T-\{[1,\emptyset], \irr\}$ if $g\leq 1$;
\item then we remove from $\wt T$ all the elements $[\tau,I]\in \wt T$ such that $[\tau-1,I]\not \in \wt T$ and $[\tau+1,I]\not \in \wt T$.
\end{itemize}  
\item A subset $T\subset T_{g,n}$ is said to be \emph{minimal} if $T=\{\irr\}$ and $g\geq 2$ or $T=\{[\tau, I],[\tau+1,I]\}$ (which then forces $g\geq 2$ or $g=1$ and $n\geq 2$) for some element $[\tau, I]\neq [1,\emptyset]$ of $T_{g,n}$. 
\end{enumerate}
\end{definition}
Observe that the empty set is admissible and it is the unique admissible subset if $g=0$ or if $(g,n)=(1,0)$. If $g\geq 2$ or $g=1$ and $n\geq 2$, then the minimal subsets are exactly the smallest admissible non-empty subsets of $T_{g,n}$. Moreover, a subset $T\subset T_{g,n}$ is admissible if and only if it is the union of the minimal subsets contained in $T$.

\begin{proposition}\label{P:equaT}
Given two subsets $T,S\subseteq T_{g,n}$, we have that 
$$\M_{g,n}^T\subseteq \M_{g,n}^S \Longleftrightarrow  T^{\adm}\subseteq S^{\adm}.$$ 
In particular, we have that $\M_{g,n}^T= \M_{g,n}^S  \Longleftrightarrow  T^{\adm}= S^{\adm}$.

\end{proposition}
\begin{proof}
We will divide the proof in four steps.

\un{Step I}: If $\{[1,\emptyset]\}\in T$ and we let $\wt T:=T-\{[1,\emptyset]\}$ then we have that 
$$
\M_{g,n}^T =\M_{g,n}^{\wt T}.
$$ 

Indeed, the $n$-pointed curves that belong to $\M_{g,n}^T\setminus \M_{g,n}^{\wt T}$ are those $n$-pointed curves $(C, \{p_i \})$ containing a tacnode $p\in C$ with $\type(p)=\{[1,\emptyset], [2,\emptyset]\}$ (in particular, if $\{[2,\emptyset]\}\not\in T$ then 
$\M_{g,n}^T =\M_{g,n}^{\wt T}$). However, if $p$ is such a tacnode then the normalisation of $C$ at $p$ will have one connected component $D$ of arithmetic genus one and without marked points. From the ampleness of $\omega_C(\sum_{i=1}^n p_i)$ it follows that either $D$ is an irreducible curve or $D$ has two irreducible components $E$ and $R$ of arithmetic genera, respectively, $1$ and $0$, meeting in a node $q$ and such that $p\in R$. In the first case, $(E,p)$ is an $A_3$-attached elliptic tail of $(C, \{p_i \})$ while in the second case $(E,q)$ is an $A_1$-attached elliptic tail of 
 $(C, \{p_i \})$.  However, both cases are impossible because if $(C, \{p_i \})\in \M_{g,n}^T$ then it cannot contain  $A_1$-attached or $A_3$-attached elliptic tails.  Hence, we conclude that $\M_{g,n}^T =\M_{g,n}^{\wt T}$.

\un{Step II}: If $g\leq 1$ and $\irr \in T$, then if we let $\wt T:=T-\{\irr\}$ then we have that 
$$
\M_{g,n}^T =\M_{g,n}^{\wt T}.
$$ 

Indeed, this follows immediately from the fact that if there exists a curve $(C,\{p_i\})\in \M_{g,n}^T$ having a tacnode of type $\{\irr\}$ then $g\geq 2$.  

\un{Step III}: For any $T\subseteq T_{g,n}$, we have that 
$$
\M_{g,n}^T =\M_{g,n}^{T^{\adm}}.
$$ 

Indeed, by Step I e II above, we can assume, up to replacing $T$ with $\wt T:=T-\{[1,\emptyset]\}$ if $g\geq 2$ and with $\wt T:=T-\{[1,\emptyset], \irr\}$ if $g\leq 1$,  that $[1,\emptyset]\not\in T$ and also that $\irr\not\in T$ if $g\leq 1$. It is then enough to show that if $[\tau,I]$ is an element of $T$ such that $[\tau-1,I]\not\in T$ and $[\tau+1,I]\not\in T$, then if we set $\wh T=T-\{[\tau, I]\}$ then we have that 
$$
\M_{g,n}^T =\M_{g,n}^{\wh{T}}.
$$ 
This is true because, given an $n$-pointed curve $(C, \{p_i \})\in \M_{g,n}^T$ , the type of a tacnode  cannot contain $[\tau, I]$ for otherwise it would contain either $[\tau-1, I]$ or $[\tau+1, I]$, which however do not belong to $T$ by assumption. Hence, the $n$-pointed curve  $(C, \{p_i \})$ 
belongs to $\M_{g,n}^{\wh T}$.

\un{Step IV}: Given $T$ and $S$ admissible subsets of $T_{g,n}$, we have that 
$$\M_{g,n}^T\subseteq \M_{g,n}^S  \Longleftrightarrow  T\subseteq S.$$

The implication $\Leftarrow$ is clear. In order to show the implication $\Rightarrow$, we will show that if $T\not\subseteq S$ then $\M_{g,n}^T\not\subseteq \M_{g,n}^S$. Since $T\not\subseteq S$, then either $\irr\in T-S$ or $[\tau,I]\in T-S$ for some $[\tau, I]\in T_{g,n}$. 

If $\irr\in T-S$ (which forces $g\geq 2$ because $T$ is admissible), then consider an $n$-pointed irreducible curve $(C,\{p_i\}\}$ of arithmetic genus $g$ having a unique singular point $p\in C$ which is furthermore a tacnode: such a curve exists in any genus $g\geq 2$ and for any $n\geq 0$, and it belongs to $\M_{g,n}^T\setminus \M_{g,n}^S$.

If, instead, $[\tau,I]\in T-S$ for some $[\tau, I]\in T_{g,n}$ then, since $T$ is admissible, we must have that $[\tau,I]\neq [1,\emptyset]$ , and either $[\tau+1,I]\in T$ or $[\tau-1,I]\in T$. Suppose for simplicity that $[\tau+1,I]\in T$ (which then forces $(\tau,I)\neq (g-1,[n])$); the other case is treated similarly by replacing $\tau$ with $\tau-1$ in what follows. Consider an $n$-pointed  curve $(C,\{p_i\}\}$  having two irreducible smooth components $D_1$ and $D_2$ meeting in one tacnode $p$, and such that $D_1$ has genus $\tau$ and contains the marked points $\{p_1\}_{i\in I}$ while $D_2$ has genus $g-\tau-1$ and contains the marked points $\{p_i\}_{i\in I^c}$. Observe that $C$ has arithmetic genus $g$, the line bundle $\omega_C(\sum_i p_i)$ is ample because $(\tau, I), (g-1-\tau,I^c) \neq (0,\emptyset)$, and $C$ does not contain 
$A_3$-attached elliptic tails because $(\tau,I), (g-1-\tau, I^c)\neq (1,\emptyset)$. Moreover, since $\type(p)=\{[\tau,I], [\tau+1, I]\}\subseteq T$ and $[\tau,I]\not\in T$, we get that 
$(C,\{p_i\}\}\in \M_{g,n}^T\setminus \M_{g,n}^S$.  
\end{proof}

\begin{remark}\label{R:number}
It follows from \cite[Lemma 3.12]{CTV1} that the number of admissible subsets of $T_{g,n}$ is the same as the number of subfaces of the elliptic bridge face, which by \cite[Rmk. 3.10]{CTV1} is equal to

$$
\begin{cases}
1 & \text{ if } g=0,\\
2 & \text{ if } (g,n)=(2,0), \\
2^{\frac{g-1}{2}} & \text{ if } n=0 \text{ and } g\geq 3 \: \text{ is odd, }\\
2^{\frac{g}{2}-1} & \text{ if } n=0 \text{ and } g\geq 4 \: \text{ is even, }\\
2^{g 2^{n-1}-1} & \text{ if } g\geq 1 \: \text{ and } n\geq 1.
\end{cases}
$$
This corresponds to the number of contractions $f_T :  \MM_{g,n}^{\ps} \to \MM_{g,n}^T$ given by Theorem \ref{T:goodsp}(\ref{T:goodsp2}).

\end{remark}

For later use, we need to recall from \cite{CTV1} a description of the closed points of the stack $\M_{g,n}^T$.

\begin{definition}\label{D:ell}[$A_1$/$A_1$-attached bridges and their types](see \cite[Def. 1.1, 1.2, 1.6]{CTV1})
\noindent 
\begin{enumerate}[(i)]

\item An \emph{elliptic bridge} is a $2$-pointed curve $(E,q_1,q_2)$ of arithmetic genus $1$ which is either irreducible  or it has two  rational smooth components $R_1$ and $R_2$ that meet in either two nodes or one tacnode   and such that $q_i\in R_i$ for $i=1,2$.
The unique elliptic bridge  containing a tacnode is called the \emph{tacnodal elliptic bridge.}

\item Let $(C, \{p_i \}_{i=1}^n)$ be an $n$-pointed curve of genus $g$.  We say that $(C, \{p_i \}_{i=1}^n)$ has an \emph{$A_{1}/A_{1}$-attached elliptic bridge}  if there exists a finite morphism $\gamma: (E,q_1,q_2) \to (C, \{p_i \}_{i=1}^n)$ (called gluing morphism) such that:
\begin{enumerate}[(a)]
\item $(E,q_1,q_2)$ is an elliptic bridge;
\item  $\gamma$ induces an open embedding of $E-\{q_1,q_2\}$ into $C- \cup_{i=1}^n \{p_i \}$; 
\item $\gamma(q_i)$ is an $A_{1}$-singularity or  a marked point (for  $i=1,2$).
\end{enumerate}
An $A_1/A_1$-attached elliptic bridge $\gamma: (E,q_1,q_2) \to (C, \{p_i \}_{i=1}^n)$  such that $\gamma(q_1)=\gamma(q_2)$ is called \emph{closed}. In this case $\gamma$ is surjective and $(g,n)=(2,0)$.

\item Let $(C, \{p_i \}_{i=1}^n)$ be a $n$-pointed curve such that $C$ is Gorenstein and $\omega_C(\sum_{i=1}^n p_i)$ is ample and let $\gamma:(E,q_1,q_2)\to (C, \{p_i \}_{i=1}^n)$ be an $A_{1}/A_{1}$-attached elliptic bridge. We say that $(E,q_1,q_2)$ is of type:
\begin{itemize}
\item $\type(E,q_1,q_2):=\{[0,\{p_i\}],[1,\{p_i\}]\} \subseteq T_{g,n}$ if either $\gamma(q_1)=p_i$ or $\gamma(q_2)=p_i$;

\item $\type(E,q_1,q_2):=\{\irr\}\subseteq T_{g,n} $ if $\gamma(q_1)$ and $\gamma(q_2)$ are singular points (either nodes or tacnodes) of $C$ and $\ov{C \setminus \gamma(E)}$ is connected (which includes also the case of a closed $A_{1}/A_{1}$-attached elliptic bridge, in which case $\ov{C \setminus \gamma(E)}=\emptyset$);
\item  $\type(E,q_1,q_2):=\{[\tau,I],[\tau+1,I]\} \subseteq T_{g,n}$ if  $\gamma(q_1)$ and $\gamma(q_2)$ are  singular points (either nodes or tacnodes) of $C$ and $\ov{C \setminus \gamma(E)}$ consists of two connected component, one of  which has arithmetic genus  $\tau$ with marked points $\{p_i\}_{i \in I}$.
\end{itemize}

\end{enumerate}
\end{definition}

Note that a  tacnodal elliptic bridge is the same thing as an  open rosary of length $2$ in the sense of \cite[Def. 1.3]{CTV1} and, therefore, it carries an action of $\Gm$ described explicitly in \cite[Rmk. 1.4]{CTV1}.

\begin{proposition}\label{P:Tclosed}(see \cite[Prop. 1.24]{CTV1})
Fix a subset $T\subset T_{g,n}$ and assume that $(g,n)\neq (2,0)$ and $\car(k)\neq 2$.

A curve $(C, \{p_i \})$ is a closed point of $\M_{g,n}^T$ if and only if $(C, \{p_i \})$ is $T$-closed, i.e. if there exists a decomposition (called the $T$-canonical decomposition) $(C,\{p_i\})=K \cup (E_1,q_1^1,q_2^1) \cup \cdots \cup (E_r,q^r_1,q^r_2)$, such that 

\begin{enumerate}[(i)]

\item  \label{def:Tclosed1} $ (E_1,q_1^1,q_2^1), \ldots, (E_r,q^r_1,q^r_2)$ are $A_1$/$A_1$-attached tacnodal elliptic bridges  of type contained in $T$.

\item \label{def:Tclosed2} $K$ does not contain tacnodes nor  $A_1$/$A_1$-attached elliptic bridges of type contained in $T$. In particular, every connected component of $K$  is a pseudo-stable curve that does not contain any 
$A_1$/$A_1$-attached elliptic bridge of type contained in $T$.

\end{enumerate}

Here $K$ (which is allowed to be empty or disconnected) is regarded as a pointed curve with marked points given by the union  of $\{p_i\}_{i=1}^n \cap K$ and of $K\cap (\overline{C \setminus K})$.
\end{proposition}

The above results is false for $(g,n)=(2,0)$ and $T^{\adm}=\{\irr\}$ (the other possibility being $T^{\adm}=\emptyset$ in which case $\M_2^T=\M_2^{\ps}$ by Proposition  \ref{P:equaT}), as we now discuss.

\begin{remark}[Closed points in $\M_2^{\irr}$]\label{R:Tclosed-g2}
The curves in $\M_2^{\irr}$ are  of the following type:  smooth curve $C_{\emptyset}$,  integral curve $C_n$ with one node and geometric genus $1$,   integral curve  $C_c$ with one cusp and geometric genus $1$,   rational curve with two nodes $C_{nn}$, a rational curve $C_{nc}$ with one node and one cusp,  curve $C_{nnn}$ made of two smooth rational curves meeting in three nodes,  rational curve $C_{cc}$ with two cusps, rational curve $C_t$ with one tacnode and  curve $C_{nt}$ made of two smooth rational curves meeting in one node and one tacnode.

The isotrivial specialisation between these curves are the following ones: $C_c$ and $C_{nc}$ isotrivially specialise to $C_{cc}$ (see \cite[Thm. 1]{HL}); 
$C_n$,  $C_{nn}$,  $C_{nc}$,  $C_{nnn}$ and   $C_t$ isotrivially specialise to  $C_{nt}$ (see \cite[Lemma 1.8]{CTV1}).
Therefore the closed points of $\M_2^{\irr}$ are the smooth curves and the two curves $C_{cc}$ and $C_{nt}$. 

A picture of all the strata of $\M_2^{\irr}$ together with all the degenerations (isotrivial or not) among them can be found in Figure \ref{Fig-2}.
\end{remark}

\begin{figure}[!h]
	\begin{center} \hspace*{-4em}
		\begin{tikzpicture}[scale=0.9]
		
		\node at (6,1.5) {Strata of $\M_{2}^{\irr}$};
		\node at (-2.5,1.5) {$\dim$};
		\draw[thick] (-3,1) to (15,1);
		\draw[thick] (-2,2) to (-2,-20);
		
		\node at (-2.5,0) {3};
		\node at (-1+4,0) {$C_\emptyset=$};
		\coordinate (a) at (0+4, 0);
		\coordinate[label=right:{2}] (b) at (3+4, 0);
		\draw [very thick, in=205, out=25] (a) to (b);
		
		\draw[->, very thick] (1+4,-0.5) -- (1+4,-1.8);

		\node at (-2.5,-3) {2};
		\node at (-1+4,-3) {$C_{n}=$};

		\coordinate (x) at (0+4,-4);
		\coordinate (y) at (1+4, -2);
		\coordinate[label=right:1] (z) at (2+4, -4);
		\draw [very thick] (x) to[in=0, out=25] (y) to[in=180, out=180] (z);

		\draw[->, very thick] (1+4,-4.5) -- (2,-6);	
		
		\draw[->, very thick] (1+4,-4.5) -- (8,-6);	
		
		\draw[thick,->,
		line join=round,
		decorate, decoration={
			zigzag,
			segment length=8,
			amplitude=.9,post=lineto,
			post length=4pt
		}](2+4,-3) -- (9.6,-3) -- (12.3,-14.5-2);

		\node   at (-2.5,-7) {1};
		\node   at (-1,-7) {$C_c=$};
		
		\coordinate (x) at (0,-8);
		\coordinate (y) at (1, -6);
		\coordinate[label=right:1] (z) at (2, -8);
		\draw [very thick] (x) to[in=-90, out=25] (y) to[in=180, out=-90] (z);

		\draw[->, very thick] (1,-8.5) -- (1,-9.5-1);

		\draw[thick,->,
		line join=round,
		decorate, decoration={
			zigzag,
			segment length=8,
			amplitude=.9,post=lineto,
			post length=4pt
		}](2,-8.5)-- (4,-10) -- (6,-14.5-2);

		\node at (5,-7) {$C_{nn}=$};
		
		\coordinate (x) at (6,-8);
		\coordinate (y) at (7, -6);
		\coordinate (z) at (8, -8);
		\draw [very thick] (x) to[in=0, out=25] (y) to[in=180, out=180] (z);
		
		\coordinate (u) at (9,-6);
		\coordinate (v) at (10, -8);
		\draw [very thick] (z) to[in=0, out=0] (u) to[in=155, out=180] (v);

		\draw[->, very thick] (7,-8) -- (7,-10.5-1);	
		
		\draw[->, very thick] (7,-8) -- (2,-10-1);	
		
		\draw[->, very thick] (10,-7) -- (13,-11);	
		
		\draw[thick,->,
		line join=round,
		decorate, decoration={
			zigzag,
			segment length=8,
			amplitude=.9,post=lineto,
			post length=4pt
		}](9,-8) -- (12,-14.5-2);

		\node at (-2.5,-11-1) {0};
		\node at (-1,-11-1) {$C_{nc}=$};
		
		\coordinate (x) at (0,-12-1);
		\coordinate (y) at (1, -10-1);
		\coordinate (z) at (2, -12-1);
		\draw [very thick] (x) to[in=0, out=25] (y) to[in=180, out=180] (z);

		\coordinate (u) at (3,-10-1);
		\coordinate (v) at (4, -12-1);
		\draw [very thick] (z) to[in=-90, out=0] (u) to[in=155, out=-90] (v);

		\draw[thick,->,
		line join=round,
		decorate, decoration={
			zigzag,
			segment length=8,
			amplitude=.9,post=lineto,
			post length=4pt
		}](2,-13.3) -- (4,-14.5-2);

		\draw[thick,->,
		line join=round,
		decorate, decoration={
			zigzag,
			segment length=8,
			amplitude=.9,post=lineto,
			post length=4pt
		}](2,-13.3) -- (11.7,-15-2);

		\node   at (6.5,-11-1) {$C_{nnn}=$};
		
		\coordinate (x) at (8, 2.5-12-1);
		\coordinate (y) at (8.5, 1-12-1);
		\coordinate (w) at (9, -0.5-12-1);
		\draw [very thick] (x) to[in=45, out=0, looseness=1.5] (y) to[in=180, out=225, looseness=1.5] (w);

		\coordinate (x) at (9, 2.5-12-1);
		\coordinate (z) at (8, -0.5-12-1);
		\draw [very thick] (x) to[in=135, out=180, looseness=1.5] (y) to[in=0, out=-45, looseness=1.5] (z);

		\draw[thick,->,
		line join=round,
		decorate, decoration={
			zigzag,
			segment length=8,
			amplitude=.9,post=lineto,
			post length=4pt
		}](8.5,-13.8) -- (11.7,-14.8-2);

		\node at (12.5,-11-1) {$C_{t}=$};

		\coordinate (x) at (0+14,-9.5-1);
		\coordinate (y) at (0.5+14, -10-1);
		\coordinate (w) at (0+14, -11.5-1);
		\coordinate (z) at (0.5+14, -12.5-1);
		\draw [very thick] (x) to[in=90, out=0] (y) to[in=90, out=-90] (w) to[in=135, out=-90] (z);

		\coordinate (x) at (0+14,-9.5-1);
		\coordinate (y) at (-0.5+14, -10-1);
		\coordinate (z) at (-0.5+14, -12.5-1);
		\draw [very thick] (x) to[in=90, out=180] (y) to[in=90, out=-90] (w) to[in=35, out=-90] (z);

		\draw[thick,->,
		line join=round,
		decorate, decoration={
			zigzag,
			segment length=8,
			amplitude=.9,post=lineto,
			post length=4pt
		}](14,-14) -- (14,-14-2);

		\node at (-2.5,-15-2) {-1};
		\node at (-1+4,-15-2) {$C_{cc}=$};

		\coordinate (x) at (0+4,-16-2);
		\coordinate (y) at (1+4, -14-2);
		\coordinate (z) at (2+4, -16-2);
		\draw [very thick] (x) to[in=-90, out=25] (y) to[in=180, out=-90] (z);
		
		\coordinate (u) at (3+4,-14-2);
		\coordinate (v) at (4+4, -16-2);
		\draw [very thick] (z) to[in=-90, out=0] (u) to[in=155, out=-90] (v);

		\node at (12.5,-15-2) {$C_{nt}=$};
		
		\coordinate (x) at (-0.5+14,2-16-2);
		\coordinate (y) at (0.5+14, 0.5-16-2);
		\coordinate (w) at (0+14, -1-16-2);
		\coordinate (z) at (0.5+14, -2-16-2);
		\draw [very thick] (x) to[in=90, out=-35] (y) to[in=90, out=-90] (w) to[in=135, out=-90] (z);

		\coordinate (x) at (0.5+14,2-16-2);
		\coordinate (y) at (-0.5+14, 0.5-16-2);
		\coordinate (z) at (-0.5+14, -2-16-2);
		\draw [very thick] (x) to[in=90, out=205] (y) to[in=90, out=-90] (w) to[in=35, out=-90] (z);

		\draw [red, very thick] plot [smooth cycle] coordinates {(5,-1.2) (2,-3) (4,-7) (-1.7,-12) (6, -15.3) (14.5,-20.5) (15, -16) (12, -12.5) (11,-7)};	
		\node [red] at (7,-2) {$\m B^{\irr}$};	
		
		\draw [blue, very thick] plot [smooth cycle] coordinates {(14,-9) (16,-14.5) (14,-21) (10.5,-15.3)};	
		\node [blue] at (14,-8.5) {$\M_2^{\irr}\setminus \M_2^{\ps}$};	
		
		\end{tikzpicture}
	\end{center}
	\caption{The strata of $\M_{2}^{\irr}$. 	A straight arrow $\to$ stands for \emph{degeneration}, while a  
	 zigzag arrow $\rightsquigarrow$ stands for \emph{isotrivial generation}. 
            The red line delimits the strata belonging to $\m B^{\irr}$ (see Remark \ref{R:Tstable-g2}) while the blue line delimits the strata belonging to $\M_2^{\irr}\setminus \M_2^{\ps}$.}
		\label{Fig-2}
\end{figure}
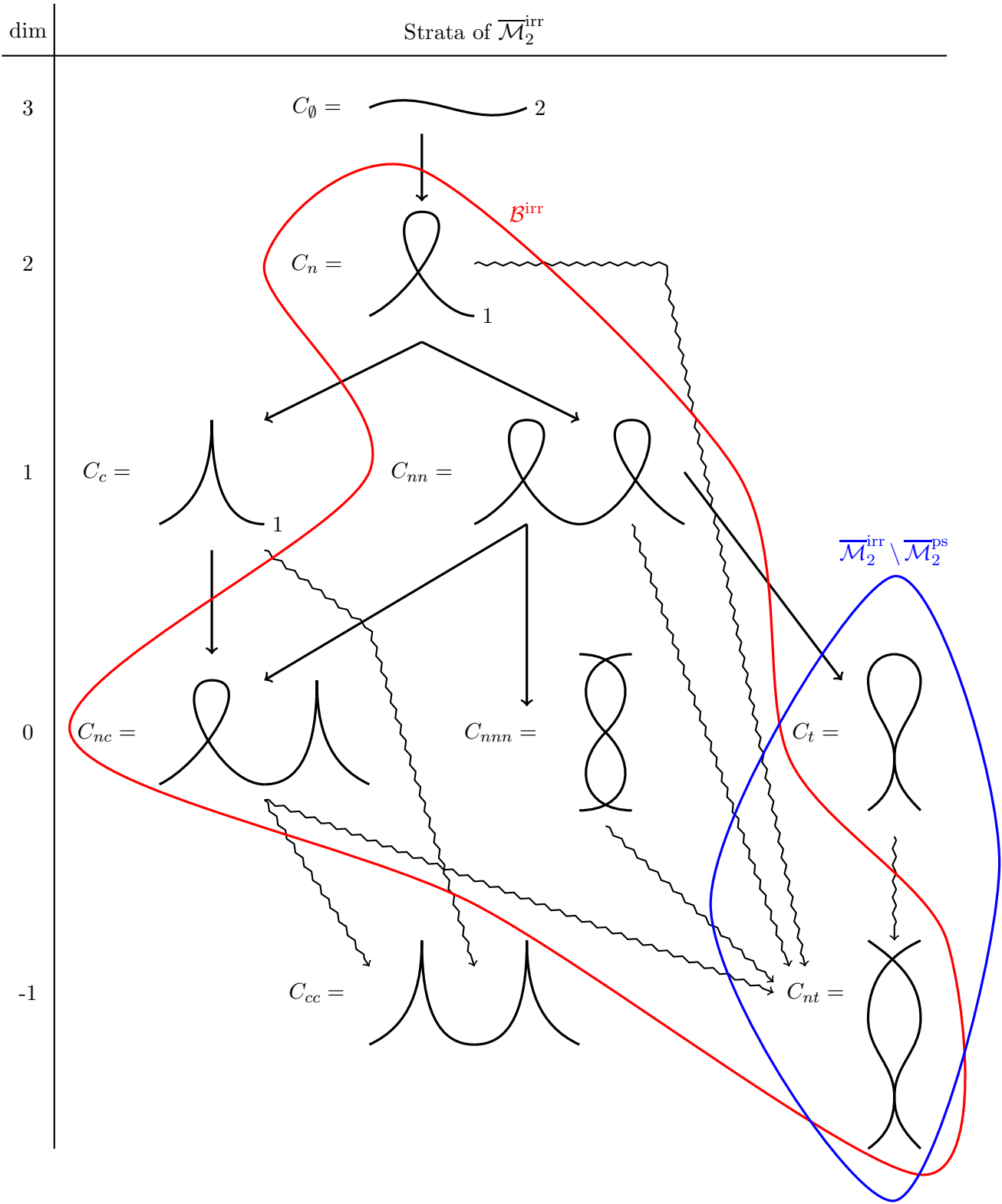

\begin{remark}\label{R:Tstable-g2}
Consider the locus $\m B^{\irr}\subset \M_2^{\irr}$  of curves of $\M_2^{\irr}$  containing an $A_1$/$A_1$-attached elliptic bridge, which is necessarily of type $\{\irr\}$ (see \cite[Def. 1.16]{CTV1}).
From the explicit description of all the points of $\M_2^{\irr}$ given in Remark \ref{R:Tclosed-g2}, it follows that $\m B^{\irr}$ is made of the curves of type $C_n$, $C_{nn}$, $C_{nc}$, $C_{nnn}$ and $C_{nt}$ (see also Figure \ref{Fig-2}). 
Hence, $\m B^{\irr}$ is not closed because it does not contain curves of type $C_c$, $C_{cc}$ or $C_t$, which are however obtained as specialisations of curves in $\m B^{\irr}$ (see Figure \ref{Fig-2}). 
Therefore, the locus  $\M_2^{\irr,+}:=\M_2^{\irr}\setminus \m B^{\irr}$ is not open in $\M_2^{\irr}$, which shows that the hypothesis $(g,n)\neq (2,0)$ is necessary in  \cite[Thm. 1.19]{CTV1}.
\end{remark}

We prove in \cite{CTV1} that the stack $\M_{g,n}^T$ admits a good moduli space $\MM_{g,n}^T$  provided that  the characteristic of the base field $k$ is big enough with respect to the pair $(g,n)$, written as $\car(k)\gg (g,n)$,  whose exact meaning is specified in \cite[Def. 2.1]{CTV1}.  

\begin{theorem}\label{T:goodsp}(\cite[Thm. 2.4, Thm. 4.1]{CTV1})
Let $(g,n)\neq (2,0)$ and fix a subset $T\subseteq T_{g,n}$. Assume  that $\car(k)\gg (g,n)$.
\begin{enumerate}
\item \label{T:goodsp1} The algebraic stack $\M_{g,n}^T$  admits a good moduli space $\MM_{g,n}^T$, which is a normal proper irreducible algebraic space over $k$. Moreover, there exists a commutative diagram 
\begin{equation}\label{E:diag-spaces}
\xymatrix{
\M_{g,n}^{\ps} \ar@{^{(}->}[r]^{\iota_T} \ar[d]^{\phi^{\ps}} & \M_{g,n}^T \ar[d]^{\phi^T} \\
\MM_{g,n}^{\ps} \ar[r]^{f_T}&   \MM_{g,n}^T   \\
}
\end{equation}
where the vertical maps are the natural morphisms to the good moduli spaces, the morphism $\iota_T$ is an open inclusion of stacks and the morphism $f_T$ is a projective   morphism. 
\item \label{T:goodsp2} If $\car(k)=0$ then $\MM_{g,n}^T$ is a projective variety and $f_T$ is the contraction of the $K$-negative face $F_T$ of the Mori cone of $\MM_{g,n}^{\ps}$, which is the convex hull of the elliptic bridge curves of type contained in $T$ (see Definition \ref{D:ell-bridg}).
\end{enumerate}
\end{theorem}

The above Theorem is false for $(g,n)=(2,0)$ and $T^{\adm}=\{\irr\}$, as we now indicate.

\begin{remark}\label{R:mod-g2}
From Remark \ref{R:Tclosed-g2},  it follows  that  the curve $C_{nc}$ can isotrivially specialise to the two distinct closed points  $C_{cc}$ and $C_{nt}$ (see also Figure \ref{Fig-2}). This implies that the stack $\M_2^{\irr}$ is not weakly separated  in the sense of \cite[Sec. 2]{ASV} and also that if a good moduli space for $\M_2^{\irr}$ exists (and we do not know if that is the case or not) then it will not be separated.
\end{remark}

In the remaining of this section, we study several geometric properties of the space $\MM_{g,n}^T$ and of the morphism $f_T:\MM_{g,n}^{\ps}\to \MM_{g,n}^T$. 

Let us start by describing the rational Picard group of $\MM_{g,n}^T$. The pull-back along the good moduli morphism $\phi^T:\M_{g,n}^T\to \MM_{g,n}^T$ induces an inclusion
\begin{equation}\label{E:pullphiT}
(\phi^T)^*:\Pic(\MM_{g,n}^T)_{\bbQ}\hookrightarrow \Pic(\M_{g,n}^T)_{\bbQ}. 
\end{equation}
Since the open inclusion $\M_{g,n}^{\ps}\subseteq \M_{g,n}^T$  has complement of codimension at least two and $\M_{g,n}^T$ is smooth,  the restriction map induces an isomorphism
\begin{equation}\label{E:PicMgT}
\Pic(\M_{g,n}^T)_{\bbQ}\xrightarrow{\cong} \Pic(\M_{g,n}^{\ps})_{\bbQ}.
\end{equation}
This implies, using Proposition \ref{P:PicMps}, that the rational Picard group $\Pic(\M_{g,n}^{T})_{\bbQ}$ is generated by the tautological line bundles 
$\{\lambda, \delta_{\irr}, \{\delta_{i,I}\}_{[i,I]\in T_{g,n}^*\setminus \{[1,\emptyset]\}}\}$.
We will now  characterise which $\bbQ$-line bundles on $\M_{g,n}^T$ belong to the image of the inclusion \eqref{E:pullphiT}. 

\begin{definition}\label{D:Tcomp}\cite[Def. 4.3]{CTV1}
A $\bbQ$-line bundle 
$$L=a\lambda+b_{\irr}\delta_{\irr}+\sum_{[i,I]\in T_{g,n}-\{[1,\emptyset], \irr\}} b_{i,I} \delta_{i,I} \in \Pic(\M_{g,n}^{T})_{\bbQ}$$
is $T$-compatible if and only if
\begin{equation}\label{E:equaT}
\begin{sis}
a+10 b_{\irr}=0 & \: \text{ if } \irr \in T, \\
a+12 b_{\irr}-b_{\tau,I}-b_{\tau+1,I}=0 & \: \text{ for any } \{[\tau, I], [\tau+1, I]\}\subset T.
\end{sis}
\end{equation}
\end{definition}
From Lemma \ref{L:int-C}, it follows that a $\bbQ$-line bundle on $\MM_{g,n}^T$ is $T$-compatible if and only if it has zero intersection with all the elliptic bridge curves of type contained in $T$ (see Definition \ref{D:ell-bridg}).

\begin{proposition}\label{P:descend}
Assume that $(g,n)\neq (2,0)$ and $\car(k)\gg (g,n)$.
A $\bbQ$-line bundle $L$ on $\M_{g,n}^T$ descends to a (necessarily unique) $\bbQ$-line bundle on $\MM_{g,n}^T$ (which we will denote by $L^T$) if and only if $L$ is $T$-compatible.
\end{proposition}

In characteristic zero, the above result follows from the Cone Theorem \cite[3.7 (4)]{KM}: since $f_T:\MM_{g,n}^{\ps}\to \MM_{g,n}^T$ is the contraction of the $K$-negative face $F_T$ of the Mori cone of $\MM_{g,n}^{\ps}$ (by Theorem \ref{T:goodsp}\eqref{T:goodsp2}), a $\bbQ$-line bundle on $\MM_{g,n}^{\ps}$  descends to a (necessarily unique) $\bbQ$-line bundle on $\MM_{g,n}^T$ if and only if it lies on $F_T^{\perp}$, i.e. if and only if it has zero intersection with all the elliptic bridge curves of type contained in $T$ (see  \cite[Cor. 4.4(i)]{CTV1}).

\begin{proof}
Up to passing to a multiple, it is enough to prove the statement for a line bundle on $\M_{g,n}^T$. Given such a line bundle $L$ on $\M_{g,n}^T$ and any one-parameter subgroup $\rho:\Gm\to \Aut(C,\{p_i\})$ for some $k$-point $(C,\{p_i\})\in \M_{g,n}^T(k)$, the group $\Gm$ will act  via $\rho$ onto the fiber $L_{(C, \{p_i\})}$ of the line bundle over $(C,\{p_i\})$ and we will denote by 
$\langle L, \rho\rangle\in \bbZ$ the weight of this action. According to \cite[Theorem 10.3]{Alper}, the line bundle $L$ descends to a $\bbQ$-line bundle on $\MM_{g,n}^T$ if and only if $\langle L, \rho\rangle=0$ for any 
one-parameter subgroup $\rho:\Gm\to \Aut(C,\{p_i\})$ of any closed $k$-point $(C,\{p_i\})\in \M_{g,n}^T(k)$. We will now show that this is the case if and only if $L$ is $T$-compatible.

To prove the if implication, assume that $L$ is $T$-compatible and fix a closed $k$-point $(C,\{p_i\})$ of $\M_{g,n}^T(k)$. By Proposition \ref{P:Tclosed}, $(C, \{p_i\})$ is $T$-closed, i.e. it admits a $T$-canonical decomposition $C=K \cup (E_1,q_1^1,q_2^1) \cup \cdots \cup (E_r,q^r_1,q^r_2)$, where $ (E_1,q_1^1,q_2^1), \ldots, (E_r,q^r_1,q^r_2)$ are $A_1$/$A_1$-attached tacnodal elliptic bridges of type contained in $T$ and  $K$ does not contain tacnodes nor $A_1$/$A_1$-attached  elliptic bridges of type contained in $T$.
 By \cite[Rmk. 1.4]{CTV1}, the connected component containing the identity of the automorphism group of  $(C, \{p_i\})$ is equal to  
 $$\Aut((C, \{p_i\}))^o=\prod_{i=1}^r \Aut((E_i,q_1^i,q_2^i))^o\cong \prod_{i=1}^r\Gm.$$  
This implies that any one-parameter subgroup of $\Aut((C, \{p_i\})$ is a linear combination of the $r$ one-parameter subgroups 
$$\rho_{E_i}:\Gm\stackrel{\cong}{\longrightarrow}\Aut((E_i,q_1^i,q_2^i))^o\subset \Aut((C, \{p_i\})).$$
 The weights $\langle L, \rho_{E_i}\rangle$ are computed in the Lemma \ref{L:weight} below. Since $\type(E_i,q_1^i,q_2^i)\subseteq T$ and  $L$ is $T$-compatible by assumption, then $\langle L, \rho_{E_i}\rangle=0$, which implies that  $\langle L, \rho\rangle=0$ for any one-parameter subgroup of $\Aut((C, \{p_i\}))$. Since this is true for any closed point $(C,\{p_i\})$ of $\M_{g,n}^T(k)$, we deduce that  $L$  descends to a $\bbQ$-line bundle on  $\MM_{g,n}^T$. 
 
 In order to prove the reverse implication, we can assume that $T$ is admissible by Proposition \ref{P:equaT} and the observation that a $\bbQ$-line bundle is $T$-compatible if and only if it is $T^{\adm}$-compatible.  Assume that $L=a\lambda+b_{\irr}\delta_{\irr}+\sum_{[i,I]\in T_{g,n}-\{[1,\emptyset], \irr\}} b_{i,I} \delta_{i,I} $  descends to a $\bbQ$-line bundle on $\MM_{g,n}^T$. For any pair $\{[\tau,I],[\tau+1,I]\}$ contained in $T$, 
 let $(D([\tau,I],[\tau+1,I]),\{p_i\})$ be the $n$-pointed curve which is  the stabilisation of the $n$-pointed curve obtained by gluing nodally a tacnodal elliptic bridge $(E,q_1,q_2)$ with a smooth curve $C_1$ of genus $\tau$ in $q_1$ and a smooth curve $C_2$ of genus $g-\tau-1$  in $q_2$ and putting the marked points $\{p_i\}_{i\in I}$ in $C_1$ and the marked points $\{p_i\}_{i\in I^c}$ in $C_2$. The curve $(D([\tau,I],[\tau+1,I]],\{p_i\})$ is $T$-closed and hence it is a closed $k$-point of $\M_{g,n}^T$ by Proposition \ref{P:Tclosed}; moreover, it has an $A_1$/$A_1$-attached elliptic tacnodal bridge $(E,q_1,q_2)$ of type $\{[\tau,I],[\tau+1,I]\}$. Since $L$ descends  to a $\bbQ$-line bundle on $\MM_{g,n}^T$ we have that $\langle L, \rho_{E}\rangle=0$ for the one-parameter subgroup $\rho_E:\Gm\stackrel{\cong}{\longrightarrow}\Aut((E,q_1,q_2))^o\subset \Aut((D([\tau,I],[\tau+1,I]),\{p_i\}))$, which translates into the equality $a+12 b_{\irr}-b_{\tau,I}-b_{\tau+1,I}=0$ by Lemma \ref{L:weight}. A similar argument can be applied to $\{\irr\}$ whenever $\irr \in T$ and it gives the equality $ a+10 b_{\irr} =0$. Hence we conclude that $L$ is $T$-compatible.
\end{proof}

\begin{lemma}\label{L:weight}
Assume that $\car(k)\neq 2$. 
Let $L=a\lambda+b_{\irr}\delta_{\irr}+\sum_{[i,I]\in T_{g,n}-\{[1,\emptyset], \irr\}} b_{i,I} \delta_{i,I} $ be a line bundle on $\M_{g,n}^T$. Let $(E,q_1, q_2)$ be an $A_1$/$A_1$-attached tacnodal elliptic bridge of a curve 
 $(C,\{p_i\})\in \M_{g,n}^T(k)$ and consider the one-parameter subgroup $\rho_E:\Gm\stackrel{\cong}{\longrightarrow}\Aut((E,q_1,q_2))^o\subset \Aut((C, \{p_i\}))$. Then we have
 $$\langle L,  \rho_E\rangle=
 \begin{cases}
  a+10 b_{\irr} & \text{ if }\type(E,q_1,q_2)=\{\irr\}, \\
a+12 b_{\irr}-b_{\tau,I}-b_{\tau+1,I} & \text{if }\type(E,q_1,q_2)=\{[\tau,I],[\tau+1,I]\}.
 \end{cases}
 $$
\end{lemma}
\begin{proof}
Since the weight is linear in $L$, the result will follow from the following identities:
\begin{equation}\label{E:weight1}
\begin{sis}
& \langle \lambda, \rho_E \rangle=1,\\
& \langle \delta_{\irr}, \rho_E\rangle=
\begin{cases}
  10 & \text{ if }\type(E,q_1,q_2)=\{\irr\}, \\
 12  & \text{if }\type(E,q_1,q_2)\neq \{\irr\},
\end{cases}\\
& \langle \delta_{i,I},\rho_E\rangle =
\begin{cases}
  -1 & \text{ if }[i,I]\in \type(E,q_1,q_2), \\
 0  & \text{if }[i,I]\not\in \type(E,q_1,q_2).
\end{cases}\\
\end{sis}
\end{equation}
The above identities can be proved by adapting the computations in \cite{AFS0}, as we now explain.

First of all, by combining \cite[Cor. 3.3]{AFS0} and the computations in \cite[Sec. 3.1.3]{AFS0} for $A_3$, we deduce that  
\begin{equation}\label{E:weightlam}
\langle \lambda, \rho_E \rangle=1.
\end{equation}

Second, in order to compute the weights of the $\psi$ classes, recall that the fiber of $\psi_i$ over a pointed curve $(C,\{p_i\}) $ is canonically isomorphic to the cotangent vector space $T_{p_i}(C)^{\vee}$. Hence, $\langle \psi_i, \rho_E\rangle$ is the weight of the action of $\Gm$, via the one-parameter subgroup $\rho_E$,  on the $1$-dimensional $k$-vector space $T_{p_i}(C)^{\vee}$. Since the action of $\Gm$ is trivial outside $E$, the weight of $\Gm$ on $T_{p_i}(C)^{\vee}$ can be non-zero only if $p_i$ belongs to $E$, in which case $p_i$ must coincide with either $q_1$ or $q_2$ and the type of $E$ must be $\{[0,\{i\}],[1,\{i\}]\}$. Moreover, if this happens, then by the explicit action of $\Gm$ on $(E,q_1,q_2)$ given in \cite[Rmk. 1.4]{CTV1}, it follows that $\langle \psi_i, \rho_E\rangle=1$. Summing up, we get that 
\begin{equation}\label{E:psiweight}
\langle \delta_{0,\{i\}}, \rho_E\rangle=-\langle \psi_i, \rho_E\rangle=
\begin{cases}
  -1 & \text{ if }[0,\{i\}]\in \type(E,q_1,q_2), \\
 0  & \text{if }[0,\{i\}]\not\in \type(E,q_1,q_2).
\end{cases}
\end{equation}

Finally, in order to compute the weights of the boundary line bundles, we will adapt the computations of \cite[Sec. 3.2.2]{AFS0}. 
Consider the (formal) semiuniversal deformation space $\Def(C,\{p_i\})$ of the $n$-pointed curve $(C,\{p_i\})$. 

Any boundary divisor $\cD$ on $\M_{g,n}^T$ restricts to a $\Gm$-invariant Cartier divisor  on $\Def(C,\{p_i\})$ given by an equation of the form  $\{f=0\}$. The $\Gm$-weight  of $f$ is equal to $-\langle \cO(\cD), \rho_E\rangle$ according to \cite[Lemma 3.11]{AFS0}. Now, since the action of $\Gm$ is trivial outside $E$, the only contributions to the weights of the boundary divisors come from the singular points lying in $E$, i.e. the tacnode $p$  of $E$ and, possibly, the two points $q_1$ and $q_2$ if they are nodes. 

In order to compute these contributions, consider the formally smooth morphism 
$$\Phi:\Def(C,\{p_i\})\longrightarrow \Def(\wh{O}_{C, p})\times \prod_{q_i \: \text{node}}\Def(\wh{O}_{C,q_i}),$$
into the product of the (formal) semiuniversal deformation spaces of the tacnode $p$ and of  nodes belonging to $\{q_1, q_2\}$. The group $\Aut(E, q_1, q_2)^o\cong \Gm$ acts on the above deformation spaces in such a way that the morphism $\Phi$ is equivariant.

Let us now write down explicitly the deformation spaces of the above singularities together with the action of $\Gm$, using the equation given in \cite[Rmk. 1.4]{CTV1}.

The  semiuniversal deformation space of $q_i$ (for $i=1,2$), whenever it is a node, is equal to $\Spf k[b_i]$ and the semiuniversal deformation family is  $w_i z_i=b_i$ where $w_i$ (resp. $z_i$) is a local coordinate on the branch of the node $q_i$ not belonging to $E$ (resp. belonging to $E$). Since the action on $\Gm$ on the local coordinates are $t\cdot (w_i)=(w_i)$ and $t\cdot (z_i)=(tz_i)$  by \cite[Rmk. 1.4]{CTV1},  the action of $\Gm$ on $\Spf k[b_i]$ is given by $t\cdot (b_i)=(tb_i)$. The locus of singular deformations of the node $q_i$ is cut out by the equation $\{b_i=0\}$, which has $\Gm$-weight one. 

On the other hand, using that $\car(k)\neq 2$, the semiuniversal deformation space of the tacnode $p$ is equal to  $\Def(\wh{O}_{C, p})\cong \Spf k[a_2,a_1,a_0]$ and the semiuniversal deformation family is given by $y^2=x^4+a_2x^2+a_1x+a_0$. 
Since the action on $\Gm$ on the local coordinates are $t\cdot x=(t^{-1}x)$ and $t\cdot (y)=(t^{-2}y)$  by \cite[Rmk. 1.4]{CTV1},  the action of $\Gm$ on $\Spf k[a_2,a_1,a_0]$ is given  by $t\cdot (a_2,a_1,a_0)=(t^{-2}a_2, t^{-3}a_1, t^{-4}a_0)$.
The locus of singular deformations of $p$ is cut out  in  $\Def(\wh{O}_{C, p})$ by the equation $\{\Delta=0\}$, where $\Delta:=\Delta(a_2,a_1,a_0)$ is the discriminant of the polynomial $x^4+a_2x^2+a_1x+a_0$. Since the discriminant is a homogeneous polynomial of degree 12 in the roots of the above polynomial and $\Gm$ acts on the roots with weight $-1$ (the same weight of $x$), it follows that $\Gm$ acts on the discriminant with weights $-12$. 

From the above discussion, it follows that the only boundary divisors of $\M_{g,n}^T$ that can have a non-zero weight against $\rho_E$ are the ones whose equation on $\Def(C,\{p_i\})$ is given by $\{0=\Phi^*(\Delta)\cdot 
\prod_{q_i \: \text{node}} \Phi^*(b_i)\}$. The Cartier divisor $\{0=\Phi^*(\Delta)\}$ comes from the restriction of $\delta_{\irr}$ to $\Def(C,\{p_i\})$, since for each generic point of $\{0=\Phi^*(\Delta)\}$ (indeed there are two generic points), the elliptic tacnodal bridge has been replaced by a nodal elliptic bridge, whose unique node is internal and hence of type $\{\irr\}$. On the other hand, depending on the types of the nodes in $\{q_1,q_2\}$, the Cartier divisor $\{0= \prod_{q_i \: \text{node}} \Phi^*(b_i)\}$ is the restriction to $\Def(C,\{p_i\})$ of the following divisor on $\M_{g,n}^T$:
\begin{itemize}
\item $2\delta_{\irr}$ if $\type(E,q_1,q_2)=\{\irr\}$; 
\item $\delta_{i, I}+ \delta_{g-1-i,I^c}=\delta_{i, I}+\delta_{i+1,I}$ if $\type(E,q_1,q_2)=\{[i,I],[g-1-i,I^c]\}=\{[i,I],[i+1,I]\}$ and $q_1$ and $q_2$ are both nodes of $C$;
\item $\delta_{1,\{i\}}$ if $\type(E,q_1,q_2)=\{[0,\{i\}],[1,\{i\}]\}$ and one among $\{q_1,q_2\}$ is a node;
\item $\cO_{\M_{g,n}^T}$ if neither $q_1$ nor $q_2$ are nodes (which can occur only if $(g,n)=(1,2)$).
\end{itemize}

We now conclude, using the above mentioned \cite[Lemma 3.11]{AFS0},  that the weights of the boundary divisors are equal to 
\begin{equation}\label{E:bouweight}
\begin{sis}
& \langle \delta_{\irr}, \rho_E\rangle=
\begin{cases}
  10 & \text{ if }\type(E,q_1,q_2)=\{\irr\}, \\
 12  & \text{if }\type(E,q_1,q_2)\neq \{\irr\},
\end{cases}\\
& \langle \delta_{i,I},\rho_E\rangle =
\begin{cases}
  -1 & \text{ if }[i,I]\in \type(E,q_1,q_2), \\
 0  & \text{if }[i,I]\not\in \type(E,q_1,q_2).
\end{cases}\\
\end{sis}
\end{equation}
By putting together \eqref{E:weightlam}, \eqref{E:psiweight} and \eqref{E:bouweight}, we deduce that \eqref{E:weight1} holds, and we are done.
\end{proof}

We now  discuss when $\MM_{g,n}^T$ is  $\bbQ$-factorial  or $\bbQ$-Gorenstein. We will first need the following

\begin{definition}\label{D:Tdiv}\cite[Def. 4.6]{CTV1}
Given a subset $T\subseteq T_{g,n}$, we define the \emph{divisorial part} of $T$ as the (possible empty) subset $T^{\div}\subset T$ defined by  
$$T^{\div}:=
\begin{cases}
\emptyset & \text{ if } (g,n)=(1,1) \, \text{ or } (2,1), \\
 \{\{[0,\{i\}],[1,\{i\}]\}: \:  \{[0,\{i\}],[1,\{i\}]\}\subset T\} & \text{ otherwise. } 
\end{cases}
$$
\end{definition}
Since $[0,\{i\}]=[1,\emptyset]$ if and only if $(g,n)=(1,1)$ and $[1,\{i\}]=[1,\emptyset]$ if and only if $(g,n)=(2,1)$, the subset $T^{\div}\subseteq T_{g,n}$ is admissible in the sense of Definition \ref{def:admissible}. Note that if $n=0$ then $T^{\div}=\emptyset$ for any subset $T$. 

\begin{proposition}\label{P:QfatQG}
Assume that $(g,n)\neq (2,0)$, $\car(k)\gg (g,n)$, and let $T\subseteq T_{g,n}$. Then the following conditions are equivalent:
\begin{enumerate}[(i)]
\item \label{P:QfatQG1} $\MM_{g,n}^T$ is $\bbQ$-factorial.
\item \label{P:QfatQG2} $\MM_{g,n}^T$ is $\bbQ$-Gorenstein.
\item \label{P:QfatQG3} $T^{\adm}=T^{\div}$.
\end{enumerate}
Under the above conditions and assuming that $(g,n)\neq (1,2), (2,1), (3,0)$, we have the following crepant equation on $\MM_{g,n}^T$ 
\begin{equation}\label{E:canT}
(\phi^T)^*(K_{\MM_{g,n}^T})=K_{\M_{g,n}^T}-8\sum_{\{[0,\{j\}], [1,\{j\}]\}\subseteq T} \delta_{1,\{j\}}=13\lambda-2\delta+\psi-8\sum_{\{[0,\{j\}], [1,\{j\}]\}\subseteq T} \delta_{1,\{j\}}.
\end{equation}

\end{proposition}

\begin{proof}
By Proposition \ref{P:equaT}, we can assume that $T=T^{\adm}$. By \cite[Prop. 4.2(iii)]{CTV1},   the morphism $f_T:\MM_{g,n}^{\ps}\to \MM_{g,n}^T$ contracts exactly $|T^{\div}|/2$ boundary divisors, namely the ones of the form $\Delta_{1,\{j\}}$ for any $1\leq j \leq n$ such that $\{[0,\{j\}], [1,\{j\}]\}\subset T$. From this and the fact that we have natural identifications 
$\Cl(\MM_{g,n}^{\ps})_{\bbQ}= \Pic(\MM_{g,n}^{\ps})_{\bbQ}\cong \Pic(\M_{g,n}^{\ps})_{\bbQ}\cong \Pic(\M_{g,n}^T)_{\bbQ}$ by Proposition \ref{P:PicMps} and \eqref{E:PicMgT}, we deduce that 
\begin{equation}\label{E:ClT}
\dim_{\bbQ} \Cl(\MM_{g,n}^T)_{\bbQ}=  \dim_{\bbQ} \Cl(\MM_{g,n}^{\ps})_{\bbQ}-\frac{|T^{\div}|}{2} =\dim_{\bbQ} \Pic(\M_{g,n}^T)_{\bbQ}-\frac{|T^{\div}|}{2}.
\end{equation}

Let us now prove the equivalence of the conditions in the statement. 

\eqref{P:QfatQG3}$\Rightarrow$ \eqref{P:QfatQG1}: Proposition \ref{P:descend} implies  that $\Pic(\MM_{g,n}^T)_{\bbQ}$ is identified, via the pull-back $(\phi^T)^*$, with the subgroup of $\Pic(\M_{g,n}^T)_{\bbQ}$ formed by $T$-compatible $\bbQ$-line bundles. Since $T=T^{\div}$ by assumption and each pair $\{[0,\{k\}], [1,\{k\}]\}\subseteq T^{\div}$ gives one relation of $T$-compatibility   (see Definition \ref{D:Tcomp}), we have that 
  \begin{equation}\label{E:PicT}
\dim_{\bbQ} \Pic(\MM_{g,n}^T)_{\bbQ}\geq \dim_{\bbQ} \Pic(\M_{g,n}^T)_{\bbQ}-\frac{|T^{\div}|}{2}.
\end{equation}
Combining \eqref{E:ClT} and \eqref{E:PicT}, we deduce that  $\dim_{\bbQ} \Cl(\MM_{g,n}^T)_{\bbQ}=\dim_{\bbQ} \Pic(\MM_{g,n}^T)_{\bbQ}$, i.e. that $\MM_{g,n}^T$ is $\bbQ$-factorial.

\eqref{P:QfatQG1}$\Rightarrow$ \eqref{P:QfatQG2}: obvious.

\eqref{P:QfatQG2}$\Rightarrow$ \eqref{P:QfatQG3}:
First, we assume that $\MM_{g,n}^T$ is $\bbQ$-Gorenstein and that $(g,n)\neq (1,2), (2,1), (3,0)$, and we prove formula \eqref{E:canT}. By the commutative diagram \eqref{E:diag-spaces}  and using the identification $\Pic(\M_{g,n}^T)_{\bbQ}\cong \Pic(\M_{g,n}^{\ps})_{\bbQ}\cong \Pic(\MM_{g,n}^{\ps})_{\bbQ}$ by Proposition \ref{P:PicMps} and \eqref{E:PicMgT}, we have that  $(\phi^T)^*(K_{\MM_{g,n}^T})=f_T^*(K_{\MM_{g,n}^T})$.
Since, as discussed above, the exceptional divisors of $f_T$ are $\{\Delta_{1,\{a_i\}}\}_{i=1}^k$ for some $\{a_1,\ldots, a_k\}\subseteq [n]$, we can write 
\begin{equation}\label{E:canT1}
(\phi^T)^*(K_{\MM_{g,n}^T})=f_T^*(K_{\MM_{g,n}^T})=K_{\MM_{g,n}^{\ps}}-\sum_{i=1}^k \gamma_i \cdot \delta_{1,\{a_i\}},
\end{equation}
for some $\gamma_i\in \bbQ$. Using our assumptions on $(g,n)$, Proposition \cite[Prop. 3.1(ii)]{CTV1} and Mumford formula (see e.g. \cite[Fact 1.28(ii)]{CTV1}) imply that 
$$K_{\MM_{g,n}^{\ps}}=K_{\M_{g,n}^{\ps}}=K_{\M_{g,n}^T}=13\lambda-2\delta+\psi.$$
Substituting into \eqref{E:canT1}, we get the formula 
\begin{equation}\label{E:canT2}
(\phi^T)^*(K_{\MM_{g,n}^T})=f_T^*(K_{\MM_{g,n}^T})=13\lambda-2\delta+\psi-\sum_{i=1}^k \gamma_i \cdot \delta_{1,\{a_i\}}.
\end{equation}
For any $1\leq j \leq k$, consider the elliptic bridge curve $C_j:=C([0,\{a_j\}], [1,\{a_j\}])$, see Definition \ref{D:ell-bridg}.   From \cite[Prop. 4.2]{CTV1}, it follows that $C_j$ is contracted by $f_T$.
Hence, by projection formula, we have that 
\begin{equation}\label{E:projfor}
0=C_j\cdot f_T^*(K_{\MM_{g,n}^T})=C_j\cdot \left(13\lambda-2\delta+\psi-\sum_{i=1}^k \gamma_i \cdot \delta_{1,\{a_i\}}\right).
\end{equation}
Now Lemma \ref{L:int-C} gives that 
\begin{equation}\label{E:intCj}
C_j\cdot  \left(13\lambda-2\delta+\psi\right)=-8 \quad \text{ and } \quad 
C_j\cdot \delta_{1,\{a_i\}}=
\begin{cases} 
-1 & \text{ if } i=j, \\
0 & \text{ otherwise.} \\
\end{cases}
\end{equation}
Substituting \eqref{E:intCj} into \eqref{E:projfor}, we get that $\gamma_i=8$ for every $1\leq i \leq k$; hence  \eqref{E:canT} is proved.  

Now we can prove that $T=T^{\adm}=T^{\div}$ under the assumption that $(g,n)\neq (1,2), (2,1), (3,0)$. Indeed, by contradiction, suppose that this is not the case. Then $T$ contains either $\{\irr\}$ or a pair $\{[\tau, I], [\tau+1,I]\}$ that is different from $\{[0,\{j\}, [1,\{j\}]\}$ for any $1\leq j\leq n$.  In any of these cases, one of the conditions \eqref{E:equaT} does not hold for  the line bundle $13\lambda-2\delta+\psi-8 \sum_{i=1}^k \delta_{1,\{a_i\}}$.  But then, by formula \eqref{E:canT}, this means that $(\phi^T)^*(K_{\MM_{g,n}^T})$ is not $T$-compatible and this is absurd by Proposition \ref{P:descend}.

It remains to deal with the special cases $(g,n)= (1,2), (2,1), (3,0)$, where formula \eqref{E:canT} is false (see \cite[Remark  3.4]{CTV1}). 
The case $(g,n)=(1,2)$ is easy since  $T^{\adm}=T^{\div}$ for any $T$.
Assume now that $(g,n)=(2,1)$ or $(3,0)$. In each of these cases, $T^{\div}=\emptyset$ while $T=T^{\adm}=\emptyset$ or $\{\irr\}$. By contradiction, assume that $T=\{\irr\}$. By \cite[Prop. 4.2(iii)]{CTV1}, the morphism $f_T$ is small and hence $f_T^*(K_{\MM_{g,n}^T})=K_{\MM_{g,n}^{\ps}}.$
Moreover,  \cite[Remark  3.4]{CTV1} implies  that 
$$f_T^*(K_{\MM_{g,n}^T})=K_{\MM_{g,n}^{\ps}}=13\lambda-2\delta+\psi-R,$$
where $R$ is an effective divisor not contained in the boundary of $\MM_{g,n}^{\ps}$. Consider now the elliptic bridge curve $C(\irr)$, see Definition \ref{D:ell-bridg}. From \cite[Prop. 4.2]{CTV1}, it follows that $C(\irr)$ is contracted by $f_T$.
  Hence, by projection formula, we have that $C(\irr)\cdot f_T^*(K_{\MM_{g,n}^T})=0$. Moreover, $C(\irr)\cdot R\geq 0$ since $C(\irr)$ is not contained in $R$, being entirely contained in the boundary. Using these facts and Lemma \ref{L:int-C}, we compute 
\begin{equation*}
0=C(\irr)\cdot f_T^*(K_{\MM_{g,n}^T})=C(\irr)\cdot \left(13\lambda-2\delta+\psi-R\right)\leq C(\irr)\cdot \left(13\lambda-2\delta+\psi\right) =-7,
\end{equation*}
which is the desired contradiction. 
\end{proof}

We finally describe a factorisation of the morphism $f_T$ into a divisorial contraction and a small contraction.
 
\begin{proposition}\label{P:factor}
Assume that $(g,n)\neq (2,0)$, $\car(k)\gg (g,n)$ and let $T\subseteq T_{g,n}$. The morphism $f_T:\MM_{g,n}^{\ps}\to \MM_{g,n}^T$ can be factorised as follows 
\begin{equation}\label{E:fact-fT}
f_T\colon \MM_{g,n}^{\ps}\xrightarrow{f_{T^{\div}}}\MM_{g,n}^{T^{\div}}\xrightarrow{\sigma_T} \MM_{g,n}^T
\end{equation}
in such a way that 
\begin{enumerate}[(i)]
\item \label{P:factor1} The morphism $f_{T^{\div}}$ is a composition of $\frac{1}{2}|T^{\div}|$ divisorial contractions, each one of them having the relative Mori cone generated by a $K$-negative extremal ray. 
\item \label{P:factor1bis} The algebraic space $\MM_{g,n}^{T^{\div}}$ is $\bbQ$-factorial and, if $\car(k)=0$, klt.

\item \label{P:factor2} The morphism $\sigma_T$ is a small contraction.
\item \label{P:factor3} 

The relative Mori cone of $\sigma_T$ is  a $K_{\MM_{g,n}^{T^{\div}}}$-negative face if and only if $T$ does not contain subsets of the form $\{[0,\{j\}],[1,\{j\}],[2,\{j\}]\}$ for some $j\in [n]$ or  $(g,n)=(3,1), (3,2), (2,2)$. 
\end{enumerate}

\end{proposition}
Note that, if $\car(k)=0$, then all the spaces appearing in \eqref{E:fact-fT} are projective varieties, and hence $f_{T^{\div}}$ is the composition of divisorial contractions of $K$-negative rays while $\sigma_T$ is a small contraction of a $K$-negative face if and only the condition on $T$ appearing in \eqref{P:factor3} is satisfied.

\begin{proof}
The open inclusions of stacks 
$$
\M_{g,n}^{\ps}\hookrightarrow \M_{g,n}^{T^{\div}}\hookrightarrow \M_{g,n}^T
$$
induce the requested factorisation of $f_T$ by passing to the good moduli spaces. Let us show that the morphisms $f_{T^{\div}}$ and $\sigma_T$ have the required properties.

Part \eqref{P:factor1}: note that we can assume $T^{\div}\neq \emptyset$, for otherwise $f_{T^{\div}}=\id$ and there is nothing to prove. 
For $i\in[n]$, let $T_i=\{[0,\{i\}],[1,\{i\}]\}$; we can write $T^{\div}=\bigcup_{i=1}^k T_{a_i}$, with $a_i\in [n]$ and $k=\frac{1}{2}|T^{\div}|$. It is also convenient to let $T^j=\bigcup_{i=1}^j T_{a_i}$ for $1\leq j\leq k$, and $T^0:=\ps$. We have  open embedding of stacks
$$\M_{g,n}^{\ps}\subsetneq \M_{g,n}^{T^1}\subsetneq  \M_{g,n}^{T^2}\subsetneq \cdots\subsetneq \M_{g,n}^{T^k}=\M_{g,n}^{T^{\div}}.$$
We denote by $f_{j+1}: \MM_{g,n}^{T^{j}}\to  \MM_{g,n}^{T^{j+1}}$ the morphism induced on the good moduli spaces by the inclusion $\M_{g,n}^{T^{j}}\subsetneq  \M_{g,n}^{T^{j+1}}$.
Note that   $f_{T^j}:=f_j\circ \ldots \circ f_1:\MM_{g,n}^{\ps}\to \MM_{g,n}^{T^j}$ and that  $ \MM_{g,n}^{T^{j}}$ is $\bbQ$-factorial (and hence $\bbQ$-Gorenstein) by Proposition \ref{P:QfatQG}.
Since $f_{T^{\div}}=f_k$, it is enough to show that each  $f_{j+1}$ is a  divisorial contraction whose relative Mori cone is generated  by a $K_{\MM_{g,n}^{T^{j}}}$-negative extremal ray of $\NEb(\MM_{g,n}^{T^{j}})$. 

First of all,  \cite[Prop. 4.2]{CTV1} implies that $f_{j+1}$ is a contraction and its exceptional locus is the divisor $\Delta_{1,\{a_{j+1}\}}$. 
Moreover, by combining Proposition \ref{P:PicMps}, \eqref{E:PicMgT},  Definition \ref{D:Tcomp} and Proposition \ref{P:descend},  we deduce that $f_{j+1}$ has relative Picard number one. Hence it remains  to take an effective curve $C$ contracted by $f_{j+1}$ and show that $K_{\MM_{g,n}^{T^{j}}}\cdot C<0$. 
By \cite[Prop. 4.2(ii)]{CTV1}, we can take as $C$ the curve  $f_{T^j}\left(C(\{[0,\{a_{j+1}\}],[1,\{a_{j+1}\}]\})\right)$, where $C(\{[0,\{a_{j+1}\}],[1,\{a_{j+1}\}]\})$ is the elliptic bridge curve of type $\{[0,\{a_{j+1}\}],[1,\{a_{j+1}\}]\}$ (see Definition \ref{D:ell-bridg}). 

Note that $(g,n)\neq (2,0)$ by hypothesis and $(g,n)\neq (2,1), (3,0)$ since we are assuming that $T^{\div}\neq \emptyset$. Moreover, if $(g,n)=(1,2)$ then  $T^{\div}=\{[0,\{1\}],[1,\{1\}]\}=\{[0,\{2\}],[1,\{2\}]\}$  
and $K_{\MM_{g,n}^{\ps}}\cdot C(\{[0,\{1\}],[1,\{1\}]\})<0$ by \cite[Prop. 3.9(i)]{CTV1}.
We can therefore assume that $(g,n)\neq (2,0), (2,1), (3,0), (1,2)$.  

By the projection formula, it is enough to show that 
\begin{equation*}\label{E:neg-int}
f_{T^j}^*(K_{\MM_{g,n}^{T^j}})\cdot C(\{[0,\{a_{j+1}\}],[1,\{a_{j+1}\}]\})<0.
\end{equation*}
Because of the assumptions on $(g,n)$,  we can apply formula \eqref{E:canT}, and we get  that 
 $$
 f_{T^j}^*(K_{\MM_{g,n}^{T^j}})=13\lambda-2\delta+\psi-8 \sum_{h=1}^j\delta_{1,a_h}\in \Pic(\MM_{g,n}^{\ps})_{\bbQ}.
 $$
Using this formula and Lemma \ref{L:int-C}, we compute that 
$$ f_{T^j}^*(K_{\MM_{g,n}^{T^j}})\cdot C(\{[0,\{a_{j+1}\}],[1,\{a_{j+1}\}]\})=-8,
$$
and we are done.

Part \eqref{P:factor1bis}: $\MM_{g,n}^{T^{\div}}$ is $\bbQ$-factorial by Proposition \ref{P:QfatQG} and, if $\car(k)=0$, it has klt singularities because $\MM_{g,n}^{\ps}$ has klt singularities (see \cite[Prop. 3.1(i)]{CTV1}) and klt singularities are preserved by divisorial contractions. 

Part \eqref{P:factor2} follows from  \cite[Prop. 4.2]{CTV1}.   

Part \eqref{P:factor3}: Let $\{S_j\}_{j\in J}$ be the collection of all minimal subset in $\left(T\setminus T^{\div}\right)$ (see Definition \ref{def:admissible}).
From \cite[Prop. 4.2(ii)]{CTV1} and \cite[Lemma 3.12(ii)]{CTV1}, it follows that the relative Mori cone $\NEb(\sigma_T)$ of $\sigma_T$ is generated by  the curves $\{(f_{T^{\div}})_*C(S_j)\}_{j\in J}$ (see Definition \ref{D:ell-bridg}). 
Therefore,  $\NEb(\sigma_T)$ is $K_{\MM_{g,n}^{T^{\div}}}$-negative if and only if 
\begin{equation}\label{E:negint}
(f_{T^{\div}})^*(K_{\MM_{g,n}^{T^{\div}}})\cdot C(S_j)=(K_{\MM_{g,n}^{T^{\div}}})\cdot (f_{T^{\div}})_*C(S_j)<0  \text{ for any } j\in J,
\end{equation}
where we used the projection formula in the first equality.

The special cases $(g,n)=(1,2), (2,1), (3,0)$ are easy to deal with: in the case $(g,n)=(1,2)$ we have that $T^{\adm}=T^{\div}$ and hence $\sigma_T$ is the identity; in the cases $(g,n)\neq (2,1), (3,0)$ then $T^{\div}=\emptyset$ which implies that $\sigma_T=f_T$. Hence we can assume that $(g,n)\neq (1,2), (2,1), (3,0)$.

Under this assumption, we can apply formula \eqref{E:canT} to get that 
\begin{equation}\label{E:pullcan}
(f_{T^{\div}})^*(K_{\MM_{g,n}^{T^{\div}}})=13\lambda-2\delta+\psi-8\sum_{\{[0,\{i\}], [1,\{i\}]\}\subseteq T} \delta_{1,\{i\}}.
\end{equation}
Using this formula and Lemma \ref{L:int-C} , we get
$$
(f_{T^{\div}})^*(K_{\MM_{g,n}^{T^{\div}}})\cdot C(S_j)=
\begin{sis}
1  & \textrm{ if  }  S_j=\{[1,\{i\}],[2,\{i\}]\} \textrm{  and  } [0,\{i\}]\in T \textrm{  for some  } i,  \\
& \textrm{ (which implies that }(g,n)\neq (3,1), (3,2), (2,2) )\\
-7 & \textrm{   otherwise.}
\end{sis}
$$
This concludes the proof of part \eqref{P:factor3}.
\end{proof}

In Proposition \ref{P:factor} we considered the contraction of the $K$-negative extremal face $F_T$, and showed that it can be decomposed into a sequence of elementary divisorial contractions which correspond exactly to the divisorial extremal rays of $F_T$,  followed by a small contraction. In general contractions of extremal faces can behave in a more subtle way as the following example shows.

%
%
%
%
%
%
%
%

\begin{example}
	There exists  a terminal 3-fold $X$ with a $K_X$-negative extremal face $F \subset \overline{NE}(X)$ generated by two extremal rays $R_1$, $R_2$ such that the contraction of $F$ is divisorial, but the contractions associated to $R_1$ and $R_2$ are both small (see \cite[Example 3.1.9]{Matsukibook} for an explicit example of this kind).
	In this case the morphism $f: X \to Y$ associated to $F$ can not be decomposed into an elementary divisorial contraction followed by a small contraction. 
\end{example}

 \section{Moduli spaces as ample models}\label{Sec:ampl-mod}

 In this section we are interested in determining $\bbQ$-divisors $L$ on $\MM_{g,n}$ (resp. on $\MM_{g,n}^{\ps}$) such that the variety $\MM_{g,n}^T$ is the projectivization of the ring of sections of $L$, i.e. 
 \begin{equation}\label{E:Projring}
 \MM_{g,n}^T= \Proj \bigoplus_{m \ge 0} H^0(\MM_{g,n}, \lfloor mL \rfloor) \quad (\text{resp. } \MM_{g,n}^T= \Proj \bigoplus_{m \ge 0} H^0(\MM_{g,n}^{\ps}, \lfloor mL \rfloor)).
 \end{equation}
 
 More precisely, we would like to understand when the morphism $\Upsilon_T:=f_T\circ \Upsilon: \MM_{g,n}\to \MM_{g,n}^T$ (resp. $f_T: \MM_{g,n}^{\ps}\to \MM_{g,n}^T$) is the ample model of $L$. 
 
 Let us recall the definition of ample model for big divisors (see \cite[Section 3.6]{BCHM10} or \cite[Def.\ 2.3]{KKL16}). Let $f:X \dashrightarrow Y$ be a birational map between normal projective varieties. Assume that  $f^{-1}$ does not contract any divisor and let $L$ be a $\bbQ$-Cartier divisor on $X$ such that $f_*L$ is also $\bbQ$-Cartier. The map $f$ is called \emph{$L$-non-positive} if for some common resolution $p:W \to X$ and $q:W \to Y$, we may write
 $$
 p^*L=q^*(f_*L) + E,
 $$
 where $E \ge 0$ is $q$-exceptional. 
 The map  $f:X \dashrightarrow Y$ is called an \emph{ample model} for $L$ if it is $L$-non-positive and $f_*L$ is ample.

 If it exists, an ample model $f: X \dashrightarrow Y$ is unique  and given by
 $$
 Y= \Proj \bigoplus_{m \ge 0} H^0(X, \lfloor mL \rfloor)
 $$
 with the induced natural map (cf.\ \cite[Lemma 3.6.6(1)]{BCHM10} or \cite[Remark 2.4(ii)]{KKL16}). The converse is also true, i.e.\ if $L$ is big and the ring of sections of $L$ is finitely generated then the induced map to its projectivization is the ample model of $L$ (see \cite[Theorem 4.2]{KKL16}). A special case of the above situation is when $L$ is \emph{semiample}, in which case the ample model of $L$ is given by the regular contraction induced by $|mL|$ for $m$ sufficiently divisible (such a morphism is called the regular contraction associated to $L$).

 In dealing with the above questions, we will often restrict ourselves to special $\bbQ$-divisors on $\MM_{g,n}$ (resp. $\MM_{g,n}^{\ps}$). We are going to use freely Corollary \ref{F:PicMM} and Proposition \ref{P:PicMps} throughout this section.

 \begin{definition}\label{D:adj}
 	We say that a $\bbQ$-divisor $L$ on $\MM_{g,n}$ (resp. $\MM_{g,n}^{\ps}$) is \emph{adjoint} if $L = K+ \psi +a\lambda+  \D$, where $K$ is the canonical divisor of stack $\M_{g,n}$ (resp. $\M_{g,n}^{\ps}$),
 	$a \ge 0$ and
 	$$
 	\D= \alpha_{\irr}\delta_{\irr} + \sum_{[i,I]\in T^*_{g,n}}\alpha_{i,I}\delta_{i,I} \quad (\text{resp. } \D= \alpha_{\irr}\delta_{\irr} + \sum_{[1,\emptyset]\neq [i,I] \in T^*_{g,n}}\alpha_{i,I}\delta_{i,I})
 	$$
 	with $0\le \alpha_{\irr} \le 1$ and $0\le \alpha_{i,I} \le 1$. 
 \end{definition}
 
  Using the formula $K=13\lambda-2\delta+\psi$, we can write adjoint $\bbQ$-divisors on $\MM_{g,n}$ (resp. $\MM_{g,n}^{\ps}$) in the following form 
 \begin{equation}\label{E:adj-div}
 \begin{aligned}
 L= (13+a)\lambda -(2-\alpha_{\irr})\delta_{\irr} - \sum_{[i,I]\in T^*_{g,n}}(2-\alpha_{i,I})\delta_{i,I}  \\
 (\text{resp. } L= (13+a)\lambda -(2-\alpha_{\irr})\delta_{\irr} - \sum_{[1,\emptyset]\neq [i,I]\in T^*_{g,n}}(2-\alpha_{i,I})\delta_{i,I}).
 \end{aligned}
 \end{equation}
  
 \begin{remark}\label{R:poladj}
 	 	Given an adjoint $\bbQ$-divisor $L = K+ \psi +a\lambda+  \D$ on $\MM_{g,n}$ as in the above definition, the pair (in the category of DM stacks) 
 	\begin{equation}\label{E:kltpair}
 	(\M_{g,n}, \Delta':=\alpha_{\irr}\Delta_{\irr} + \sum_{\stackrel{[i,I]\in T_{g,n}^*:}{|I|\geq 2\: \text{ if } \: i=0}} \alpha_{i,I}\Delta_{i,I}),
 	\end{equation}
 	is lc (log canonical) since the boundary divisor of $\M_{g,n}$ is a normal crossing divisor and all the coefficients of $\Delta'$ are non-negative and strictly less than or equal to 1. Moreover, the $\bbQ$-line bundle $a\lambda+\sum_{j=1}^n (1-\alpha_{0,\{j\}})\psi_j$ is nef, since $\lambda$ is nef (see \cite[Chap. XIV, Lemma (5.6)]{GAC2}) and $\psi_i$ is  nef for each $i$ by \cite[Chapter XIV, Corollary (5.14)]{GAC2}. 
 	
 	Therefore,  $L$ is a \emph{polarised adjoint $\bbQ$-divisor in the sense of generalized pairs} with respect to the lc pair 
 	$(\M_{g,n}, \Delta')$  and the nef divisor $a\lambda+\sum_{j=1}^n (1-\alpha_{0,\{j\}})\psi_j$, i.e. 
 	$$
 	L=K_{\M_{g,n}}+\Delta'+a\lambda+\sum_{j=1}^n (1-\alpha_{0,\{j\}})\psi_j \, ,
 	$$
 	see \cite{Birkar-Zhang}. Our choice is to use only the term \emph{adjoint divisor} since no confusion can arise.
 	The analogous remark is true for  adjoint $\bbQ$-divisors on $\MM_{g,n}^{\ps}$.
 
 \end{remark}
 
The main result of this section is the following.

\begin{theorem}\label{T:all-one}
Assume that ${\rm char}(k)=0$ and let $L$ be an adjoint $\bbQ$ divisor on $\MM_{g,n}$ as in Definition \ref{D:adj}.
\begin{enumerate}
\item \label{T:all-one1} $L$ is ample if and only if it is F-ample. In this case, we have that  
$$ \MM_{g,n}= \Proj \bigoplus_{m \ge 0} H^0(\MM_{g,n}, \lfloor mL \rfloor).$$
\item \label{T:all-one2}  Assume that $(g,n)\neq (1,1), (2,0)$. Then $L$ is semiample with associated contraction equal to $\Upsilon:\MM_{g,n}\to \MM_{g,n}^{\ps}$
if and only if it is F-nef and the only F-curve on which it is trivial is $C_{\rm ell}$. In this case, we have that 
$$ \MM_{g,n}^{\ps}= \Proj \bigoplus_{m \ge 0} H^0(\MM_{g,n}, \lfloor mL \rfloor).$$
\item \label{T:all-one3} Fix $T\subseteq T_{g,n}$ and assume that $(g,n)\neq (1,1), (2,0), (1,2)$ and that $\alpha_{\irr}\leq \frac{10-a}{12}$.  Then 
 $\Upsilon_*(L)$ is semiample with associated contraction equal to $f_T:\MM_{g,n}^{\ps}\to \MM_{g,n}^T$ if and only if  $\Upsilon^*(\Upsilon_*(L))$ is F-nef and the only F-curves on which it is trivial are the ones whose images in $\MM_{g,n}^{\ps}$ have numerical classes contained in $F_T$. Moreover, in this case the ample model of $L$ is equal to 
$\Upsilon_T=f_T\circ \Upsilon:\MM_{g,n}\to \MM_{g,n}^T$ if we assume furthermore that 
$\displaystyle \alpha_{\irr}\leq \frac{9-a+\alpha_{1,\emptyset}}{12}$.  In particular, we have that 
$$ \MM_{g,n}^{T}= \Proj \bigoplus_{m \ge 0} H^0(\MM_{g,n}, \lfloor mL \rfloor)=\Proj \bigoplus_{m \ge 0} H^0(\MM_{g,n}^{\ps}, \lfloor m\Upsilon_*(L) \rfloor).$$
\end{enumerate}
\end{theorem}

The intersection-theoretic conditions appearing in the above theorem will be translated into explicit numerical conditions for the coefficients of $L$ in Lemmas \ref{L:Fnef} and  \ref{L:F-nef}. The following remark describes explicitly the F-curves appearing in part \eqref{T:all-one3} of the theorem.

  \begin{remark}\label{R:Fcur-T}
Fix $T\subseteq T_{g,n}$ and assume that $(g,n)\neq (1,1), (2,0), (1,2)$. 
Since the relative Mori cone $\NE(\Upsilon_T)$  is equal to $\mathrm{Ker}({\Upsilon_T}_*) \cap \NE(\MM_{g,n})$ and the relative Mori cone $\NE(f_T)$ is equal to  $\mathrm{Ker}({f_T}_*) \cap \NE(\MM_{g,n}^{\ps})$ , we have that $\NE(\Upsilon_T)=\Upsilon_*^{-1}(F_T) \cap \NE(\MM_{g,n})$.

Then the only F-curves of $\MM_{g,n}$ whose images in $\MM_{g,n}^{\ps}$ have numerical classes contained in $F_T$, or equivalently such that their numerical class belong to $\NE(\Upsilon_T)$, are (in the notation of Section \ref{S:prel}):
	\begin{itemize}
	\item $C_{\rm ell}$, 
	\item $F_s([1,\emptyset])$ if $\irr \in T$ and $g\geq 2$,
	\item  $F([\tau,I],[g-\tau-1,I^c])$ for every $\{[\tau,I], [\tau+1,I]\}\subseteq T\setminus \{[1,\emptyset]\}$.  
	\end{itemize}
This follows from an inspection of the list of F-curves (see \S \ref{S:prel}) using that an integral curve of $\MM_{g,n}$ has numerical class contained in $\NE(\Upsilon_T)$ if and only if it is either contracted by $\Upsilon$ or its image via $\Upsilon$ is an elliptic bridge curve of type contained in $T$ by \cite[Lemma 3.9(ii)]{CTV1}.
 \end{remark}

The proof of the above results will be divided in a few steps; we start off with the following remark clarifying  the relation between adjoint $\bbQ$-divisors on $\MM_{g,n}$ and on $\MM_{g,n}^{\ps}$, and their ample models.

 \begin{remark}\label{R:comp-adj}
 Assume that $(g,n)\neq (1,1), (2,0)$.
 	\noindent 
 	\begin{enumerate}[(i)]
 		\item \label{R:comp-adj1} If $L= K+\psi+a\lambda+\Delta$ is an adjoint $\bbQ$-divisors on $\MM_{g,n}$, then 
 		$$\Upsilon_*(L)=K+\psi+a\lambda+\Upsilon_*(\Delta)$$
 		is an adjoint $\bbQ$-divisor on $\MM_{g,n}^{\ps}$ and, from Proposition \ref{P:PicMps}\eqref{P:PicMps3},  we have that
		\begin{equation}\label{E:push-pull}
	L=\Upsilon^*(\Upsilon_*(L))+ (9+\alpha_{1,\emptyset}-a-12\alpha_{\irr}) \delta_{1,\emptyset}.
	\end{equation}
	Note that $L$ and $\Upsilon_*(L)$ have the same ample model if and only if $\Upsilon$ is $L$-non-positive, which happens if and only if 
	$$\alpha_{\irr}\leq \frac{9-a+\alpha_{1,\emptyset}}{12}.$$
It will be useful to notice that, since $\alpha_{1,\emptyset}\leq 1$, we also have 
$$ \frac{9-a+\alpha_{1,\emptyset}}{12}\leq \frac{10-a}{12}.$$

 		\item \label{R:comp-adj2}  If $L= K+\psi+a\lambda+\Delta$ is an adjoint $\bbQ$-divisors on $\MM_{g,n}^{\ps}$, then using Proposition \ref{P:PicMps}\eqref{P:PicMps3} and \eqref{E:adj-div} we get (for any $\beta\in \bbQ$):
 		\begin{equation}\label{E:pull-adj}
 		\begin{aligned}
 		& \Upsilon^*(L)+\beta\delta_{1,\emptyset}=(13+a)\lambda -(2-\alpha_{\irr})\delta_{\irr} - \sum_{[1,\emptyset]\neq [i,I]\in T^*_{g,n}}(2-\alpha_{i,I})\delta_{i,I}-(11-12\alpha_{\irr}-a-\beta)\delta_{1,\emptyset}= \\
 		& = K+ \psi +a\lambda+ \alpha_{\irr}\delta_{\irr} + \sum_{[1,\emptyset]\neq [i,I] \in T^*_{g,n}}\alpha_{i,I}\delta_{i,I} + (12\alpha_{\irr}+a+\beta-9)\delta_{1,\emptyset}.
 		\end{aligned}
 		\end{equation}
 		In particular, we deduce that 
 		\begin{itemize}
 			\item $\Upsilon^*(L)$ is of adjoint type if and only if $\frac{9-a}{12}\leq \alpha_{\irr} \le \frac{10-a}{12}$;
 			\item there exists $\beta\geq 0$ such that $\Upsilon^*(L)+\beta\delta_{1,\emptyset}$ is of adjoint type if and only if $\alpha_{\irr}\le \frac{10-a}{12}$.
 		\end{itemize}
 	\end{enumerate}
	Note that $\Upsilon^*(L)+\beta\delta_{1,\emptyset}$ and $L=\Upsilon_*(\Upsilon^*(L)+\beta\delta_{1,\emptyset})$ have the same ample model if and only if $\Upsilon$ is $(\Upsilon^*(L)+\beta\delta_{1,\emptyset})$-non positive, which happens if and only if $\beta\geq 0$. 
 \end{remark}

An important property of adjoint divisors on $\MM_{g,n}$ or  on $\MM_{g,n}^{\ps}$ is that they fulfil the expectations of the F-conjecture in the following sense.

\begin{proposition}\label{P:F-conjecture}
Assume ${\rm char}(k) =0$.
\begin{enumerate}
\item \label{P:F-conjecture1} Let $L$ be an adjoint $\bbQ$-divisor on $\MM_{g,n}$. If $L$ is F-ample (resp.\ F-nef) then $L$ is ample (resp.\ nef).
\item \label{P:F-conjecture2} Assume that $(g,n)\neq (1,1), (2,0)$ and let $L$ be an adjoint $\bbQ$-divisor on $\MM_{g,n}^{\ps}$ such that $\alpha_{\irr}\leq \frac{10-a}{12}$. If $\Upsilon^*(L)$ is F-nef then $\Upsilon^*(L)$ is nef (and hence $L$ is nef). 
\end{enumerate}
\end{proposition}

In proving the above Proposition, a crucial role is played by  the morphism (studied in \cite{GKM})
\begin{equation}\label{E:fGKM}
f: \MM_{0,n+g} \to \MM_{g,n}
\end{equation}
given by gluing $g$ copies of the pointed rational elliptic curve at the last $g$ marked points of a curve in $\MM_{0,n+g}$.
We will need the following Lemma, which is based on a result of Keel-McKernan \cite[Thm. 1.2]{KM13} characterising certain extremal rays of the Mori cone of $\MM_{0,N}$. 

\begin{lemma}\label{L:GKMmap}
Assume ${\rm char}(k)=0$.
Let $L$ be a $\bbQ$-divisor on $\MM_{g,n}$ of the form 
$$L = K_{\M_{g,n}}+ \psi +a\lambda+  \alpha_{\irr}\delta_{\irr} + \sum_{[i,I]\in T^*_{g,n}}\alpha_{i,I}\delta_{i,I}$$
with
\begin{equation}\label{E:ipoGKM}
\begin{sis}
& 0\le \alpha_{i,I} \le 1 \quad \text{for any  } [i,I]\neq [1,\emptyset] \:  \text{and  }[i,I]\neq [0,\{k\}] \text{ with } 1\leq k \leq n,\\
& \alpha_{1,\emptyset}\leq 1, \\
& \alpha_{0,\{k\}}\leq 1  \text{ for any  } 1\leq k \leq n.
\end{sis}
\end{equation}
If $L$ is non-negative (resp. positive) on all the F-curves of $\MM_{g,n}$ of type (6) (see \S\ref{S:prel}), then $f^*(L)$ is nef (resp. ample).
\end{lemma}
\begin{proof}
Let us first compute $f^*(L)$. First of all, we have that 
\begin{equation}\label{E:pull-can}
f^*(K_{\M_{g,n}} + \psi)= K_{\M_{0,g+n}} + \psi=K_{\MM_{0,g+n}}+\psi. 
\end{equation}
Indeed, by \cite[Chap. XIII, Thm. (7.6) and Thm. (7.15)]{GAC2}, we have the formula  $K+\psi=2\kappa_1-11\lambda$ both on $\M_{g,n}$ and on $\M_{0,g+n}=\MM_{0,g+n}$. Now the pull-back $f^*$ preserves $\kappa_1$ by \cite[Chap. XVII, Lemma 4.38]{GAC2} and it also sends $\lambda$ to zero (and hence it preserves it) because the only moving curves 
in the image of $\lambda$ are rational curves. Hence formula \eqref{E:pull-can} follows. 

Furthermore, using \cite[Chap. XVII, Lemma 4.38]{GAC2} again, we see that 
\begin{equation}\label{E:pull-del}
\begin{sis}
& f^*\delta_{\irr}=\delta_{\irr}=0, \\
&  f^*\delta_{i,I}= \sum_{\substack{J \subseteq \{n+1,\ldots, n+g\}\\ |J|=i}} \delta_{0,I \coprod J} \quad \text{ for any } [i,I] \in T^*_{g,n},
\end{sis}
\end{equation}
 In particular, the only line bundles of the form $f^*\delta_{i,I}$ that are not boundary line bundles of $\MM_{0,g+n}$ are 
 \begin{equation}\label{E:pull-notbd}
 \begin{sis}
 & f^* \psi_i=-f^*\delta_{0,\{i\}}=-\delta_{0,\{i\}}=\psi_i \quad \text{ for any } 1\leq i \leq n, \\
 & f^*\delta_{1,\emptyset} = \sum_{i=n+1}^{g+n} \delta_{0,\{i\}}=-\sum_{i=n+1}^{g+n} \psi_i.
 \end{sis}
 \end{equation}
By putting \eqref{E:pull-can}, \eqref{E:pull-del} and \eqref{E:pull-notbd} together, we get that 
\begin{equation}\label{E:pull-L}
f^*L=K_{\MM_{0,g+n}} +\sum_{\substack{[0,I\coprod J] \in T^*_{0,g+n}:\\2\leq |I|+|J|\leq g+n-2}} \alpha_{|J|,I}  \delta_{0,I \cup J}  +\left[\sum_{i=1}^n (1-\alpha_{0,\{i\}}) \psi_i +(1-\alpha_{1,\emptyset}) \sum_{j=n+1}^{n+g} \psi_j\right], 
\end{equation}
where the elements  $[0,I\coprod J] \in T^*_{0,g+n}$ are written in such a way that $I\subseteq [n]$ and $J\subseteq \{n+1,\ldots, n+g\}$. Now we define the following $\bbQ$-divisors on $\MM_{0,g+n}$
\begin{equation*}
\begin{sis}
& \Delta:=\sum_{\substack{[0,I\coprod J] \in T^*_{0,g+n}:\\2\leq |I|+|J|\leq g+n-2}} \alpha_{|J|,I}  \delta_{0,I \cup J}, \\
& N:=\left[\sum_{i=1}^n (1-\alpha_{0,\{i\}}) \psi_i +(1-\alpha_{1,\emptyset}) \sum_{j=n+1}^{n+g} \psi_j\right], 
\end{sis}
\end{equation*}
so that the expression \eqref{E:pull-L} becomes 
\begin{equation}\label{E:pull-L2}
f^*L=K_{\MM_{0,g+n}}+\Delta+N.
\end{equation}
Note that the hypothesis \eqref{E:ipoGKM} together with the fact that $\psi_i$ is nef for any $1\leq i \leq g+n$ by \cite[Chap. XIV, Cor. (5.14)]{GAC2}, implies that $\Delta$ is a boundary divisor on $\MM_{0,g+n}$ (i.e. a sum of the boundary irreducible components of $\MM_{0,g+n}$ each with coefficient in between $0$ and $1$) and that $N$ is a nef divisor.  

Now suppose by contradiction that $f^*L$ is not nef (resp. not ample). Then there exists an extremal ray $R$ of the Mori cone $\NE_1(\MM_{0,g+n})$ such that 
\begin{equation}\label{E:negL}
f^*L\cdot R<0 \quad (\text{resp. }\leq 0).
\end{equation} 
Using \eqref{E:pull-L2} and the fact that $N$ is nef, both the inequalities \eqref{E:negL} imply  the following inequality  
\begin{equation}\label{E:negD}
(K_{\MM_{0,g+n}}+\Delta)\cdot R  \leq 0.
\end{equation} 
Now we can apply \cite[Theorem 1.2(2)]{KM13} (which needs ${\rm char}(k)=0$) in order to conclude that $R$ is generated by an F-curve $C$ of $\MM_{0,g+n}$. The image $f(C)$ of this F-curve via $f$ will be an F-curve of $\MM_{g,n}$ of type (6), see \S \ref{S:prel}. Now the inequality \eqref{E:negL} together with the projection formula implies 
\begin{equation}\label{E:negL2}
L\cdot f(C)<0 \quad (\text{resp. }\leq 0),
\end{equation} 
and this contradicts the assumption that $L$ intersects non-negatively (resp. positively) all the F-curves of $\MM_{g,n}$ of type (6). 
\end{proof}

\begin{proof}[Proof of Proposition \ref{P:F-conjecture}]

Let us first prove part \eqref{P:F-conjecture1}. Note that the line bundle $L$ satisfies the numerical assumptions \eqref{E:ipoGKM} because it is adjoint on $\MM_{g,n}$ and it  intersects  non-negatively (resp. positively) all the F-curves of type (6) because it is F-nef (resp. F-ample) by assumption. Hence, we can apply  Lemma \ref{L:GKMmap} in order to infer that 
$f^*L$ is nef (resp. ample). This fact, together with the fact that $L$ intersects  non-negatively (resp. positively) all the F-curves of type (1) through (5) because it is F-nef (resp. F-ample) by assumption, implies by \cite[Cor. (4.3)]{GKM} that $L$ is nef (resp. ample). 

The proof of part \eqref{P:F-conjecture2} is similar and it uses the fact that $\Upsilon^*(L)$ satisfies  the numerical assumptions \eqref{E:ipoGKM} by the formula \eqref{E:pull-adj} (with $\beta=0$) using that $L$ is an adjoint line bundle on $\MM_{g,n}^{\ps}$ with  $\alpha_{\irr} \le \frac{10-a}{12}$.
\end{proof}

We now formulate a criterion to check whether a $\bbQ$-divisor (not necessarily adjoint) on $\MM_{g,n}$ (resp. on $\MM_{g,n}^{\ps}$) is semiample with associated contraction equal to $\Upsilon:\MM_{g,n}\to \MM_{g,n}^{\ps}$ (resp. $f_T:\MM_{g,n}^{\ps}\to \MM_{g,n}^T$ for some $T\subseteq T_{g,n}$).

\begin{lemma}\label{L:map-nef}
	Assume ${\rm char}(k) =0$ and $(g,n)\neq (1,1), (2,0)$.
	\begin{enumerate}
	\item \label{L:map-nef1} 
	Let $L$ be a $\bbQ$-divisor on $\MM_{g,n}$. Then $L$ is semiample with associated contraction equal to  $\Upsilon:\MM_{g,n}\to \MM_{g,n}^{\ps}$ if and only if $L$ is nef and trivial only on the curves whose numerical class is in $\bbR_{\geq 0}\cdot [C_{\rm ell}]$.
	\item \label{L:map-nef2}  Let $L$ be a $\bbQ$-divisor on $\MM_{g,n}^{\ps}$  and fix $T\subseteq T_{g,n}$.  Then $L$ is semiample with associated contraction equal to  $f_T:\MM_{g,n}^{\ps}\to \MM_{g,n}^T$ if and only if $L$ is nef and trivial only on the curves whose numerical class is contained in $F_T$, or equivalently if and only if $\Upsilon^*(L)$ is nef and trivial only on $\NE(\Upsilon_T)$.
	\end{enumerate}
	
\end{lemma}

\begin{proof} 
Note first that the relative Mori cone of $\Upsilon$ is equal to $\bbR_{\geq 0}\cdot [C_{\rm ell}]$ by \cite[Prop. 3.5(ii)]{CTV1} and  the relative Mori cone of $f_T$ is equal to $F_T$ by \cite[Thm. 4.1]{CTV1}.  These relative cones are K-negative faces, where $K$ is the canonical divisor of $\MM_{g,n}$ or $\MM_{g,n}^{\ps}$, by Proposition \cite[Prop. 3.5(ii)]{CTV1} and  \cite[Prop. 3.9(i)]{CTV1}.

In Case \eqref{L:map-nef1} of the statement, $L$ is a nef divisor which supports exactly $\bbR_{\geq 0}\cdot [C_{\rm ell}]$, while in Case \eqref{L:map-nef2} $L$ is a nef divisor supporting $F_T$. The result follows hence by the cone theorem \cite[Theorem 3.7]{KM} and its proof. More precisely, one sees that $mL-K$ is ample for $m \gg 0$ and so $L$ is semiample, inducing the desired contraction.

The last equivalence in part \eqref{L:map-nef2} follows from the projection formula together with the fact  that  all the curves in $\MM_{g,n}^{\ps}$ are images of curves in $\MM_{g,n}$ since $\Upsilon$ is surjective and projective. 	
\end{proof}

\begin{remark}\label{R:othercase}
A priori, we could have considered another possibility in the above Lemma, namely  those  $\bbQ$-divisors $L$ on $\MM_{g,n}$ that are semiample with associated contraction $\Upsilon_T=f_T\circ \Upsilon: \MM_{g,n}\to \MM_{g,n}^T$ for some $T\subseteq T_{g,n}$. However, in this case $L=\Upsilon^*(\Upsilon_*(L))$ and $\Upsilon_*(L)$ are semiample with associated contraction equal to $f_T:\MM_{g,n}^{\ps}\to \MM_{g,n}^T$, as in Lemma \ref{L:map-nef}\eqref{L:map-nef2}. 
\end{remark}

We now prove that for adjoint divisors, in each of the cases of Lemma \ref{L:map-nef}, it is enough to check the conditions only on F-curves of $\MM_{g,n}$.  The crucial ingredients are the positivity results proved in \cite{AFS3} for $K_{\M_{g,n}}+\psi+9/11(\delta-\psi)$ on $\MM_{g,n}$ and for $K_{\M_{g,n}^{\ps}}+\psi+7/10(\delta-\psi)$ on $\MM_{g,n}^{\ps}$.

\begin{proposition}\label{P:1strata}
	Assume ${\rm char}(k) =0$ and $(g,n)\neq (1,1), (2,0)$.
	\begin{enumerate}
	\item \label{P:1strata1} 
	Let $L$ be an adjoint $\bbQ$-divisor on $\MM_{g,n}$. Then $L$ is nef and trivial only on the curves whose numerical class is in $\bbR_{\geq 0}\cdot [C_{\rm ell}]$ if and only if $L$ is F-nef and the only F-curve on which it is trivial is $C_{\rm ell}$.

	\item \label{P:1strata2}  Let $L$ be an adjoint $\bbQ$-divisor on $\MM_{g,n}^{\ps}$  and fix $T\subseteq T_{g,n}$.  Assume that $\alpha_{\irr}\leq \frac{10-a}{12}$ and that $(g,n)\neq (1,2)$.
	
	Then $L$ is nef and trivial only on the curves whose numerical class is contained in $F_T$ if and only if $\Upsilon^*(L)$ is F-nef and the only F-curves on which it is trivial are the ones whose images in $\MM_{g,n}^{\ps}$ have numerical classes contained in $F_T$.
	\end{enumerate}
\end{proposition}
Note that the condition $\alpha_{\irr}\leq \frac{10-a}{12}$ appearing in \eqref{P:1strata2} is natural from different point of views by Remark \ref{R:comp-adj} and also quite mild as we will see in Remark \ref{R:simp-ineq}\eqref{R:simp-ineq1}.

\begin{proof}
Note that the only if implications are trivial in all the cases, hence we will focus on the if implication.

Let us first prove \eqref{P:1strata1}. Assume that $L$ is F-nef and the only F-curve on which it is trivial is $C_{\rm ell}$. We want to show that $L$ is nef and trivial only on the curves whose numerical class is in $\bbR_{\geq 0}\cdot [C_{\rm ell}]$. For that purpose, using that $K_{\M_{g,n}}+\psi+\frac{9}{11}(\delta-\psi)$  is a nef divisor on $\MM_{g,n}$ supporting the extremal ray $\bbR_{\geq 0}\cdot [C_{\rm ell}]$  by \cite[Introduction]{AFS3}, it is enough to show that the $\bbQ$-divisor 
 	$$
 	M(t):=tL-\left(K_{\M_{g,n}}+\psi+\frac{9}{11}(\delta-\psi)\right)
 	$$
 	is nef for $t \gg 0$. 
	
Note that $M(t)$ is F-nef for $t\gg 0$ since $L$ is positive on all the F-curves that are different from $C_{\rm ell}$ and $K_{\M_{g,n}}+\psi+\frac{9}{11}(\delta-\psi)$ is zero on 
$C_{\rm ell}$. Using this, it follows  from \cite[Cor. 4.3]{GKM} that $M(t)$ is nef if (and only if) its pull-back via the gluing morphism $f:\MM_{0,g+n}\to \MM_{g,n}$ of \eqref{E:fGKM} is nef. This will follow if we show that $f^*(L)$ is ample. 

In order to show this, we apply Lemma \ref{L:GKMmap}. Since $L$ is an adjoint divisor on $\MM_{g,n}$, it satisfies the numerical assumptions \eqref{E:ipoGKM}. Moreover, $L$ is positive on the F-curves of type (6) by assumption. It follows from Lemma \ref{L:GKMmap} that $f^*(L)$ is ample and we are done.

Let us finally  prove part \eqref{P:1strata2}. Assume that $\Upsilon^*(L)$ is F-nef and the only F-curves on which it is trivial are  the ones whose image in $\MM_{g,n}^{\ps}$ has numerical class contained in $F_T$. We want to show that $L$ is nef and trivial only on the curves whose numerical classes is contained in $F_T$. 
For that purpose, we will show the following 

\un{Claim:} the $\bbQ$-divisor on $\MM_{g,n}^{\ps}$
 	$$
 	tL-\left(K_{\M_{g,n}^{\ps}}+\psi+\frac{7}{10}(\delta-\psi)\right)
 	$$
 	is nef for $t \gg 0$.

Let us show that the claim will prove the desired statement. Indeed, it follows from \cite[Thm. 1.2(a)]{AFS3} that $K_{\M_{g,n}^{\ps}}+\psi+\frac{7}{10}(\delta-\psi)$ is a nef divisor on $\MM_{g,n}^{\ps}$ such the only integral curves  on which it vanishes are the elliptic bridge curves (see Section \ref{S:prel}). Therefore, this fact together with the above claim, imply that $L$ is nef and that the only integral curves on which it is possibly zero are the elliptic bridge curves. However, each elliptic bridge curve of $\MM_{g,n}^{\ps}$ is the  image of an F-curve of $\MM_{g,n}$ and, by the assumption on $\Upsilon^*(L)$, the only ones on which $L$ vanishes are the ones of type contained in $T$. This implies that $L$ is trivial only on the curves whose numerical classes are contained in $F_T$.

Let us now prove the claim.  Since any curve in $\MM_{g,n}^{\ps}$ is the image of a curve in $\MM_{g,n}$ because $\Upsilon$ is projective and surjective, it is enough (and indeed necessary)  to show that the $\bbQ$-divisor on $\MM_{g,n}$ 
 	\begin{equation}\label{E:Lt}
 	N(t):=\Upsilon^*\left(tL-\left(K_{\M_{g,n}^{\ps}}+\psi+\frac{7}{10}(\delta-\psi)\right)\right)
 	\end{equation}
 	is nef for $t\gg 0$. It follows from  \cite[Cor. 4.3]{GKM} that $\Upsilon^*(N(t))$ is nef for $t\gg 0$ if (and only if)
\begin{enumerate}[(a)]
\item \label{prop-a} $\Upsilon^*(N(t))$ is F-nef for $t\gg 0$;
\item \label{prop-b} $f^*(\Upsilon^*(N(t)))$ is nef for $t\gg 0$. 
\end{enumerate}	 
 Let us show that both these two properties hold true, which will conclude our proof.
 
 Property \eqref{prop-a} holds true because, by assumption, $\Upsilon^*(L)$ is F-nef and the only F-curves on which it vanishes are the one whose class belong to $\NE(\Upsilon_T)$, on which also  $K_{\M_{g,n}^{\ps}}+\psi+\frac{7}{10}(\delta-\psi)$ vanishes as recalled above.
 
In order to show property \eqref{prop-b}, it is enough to prove that $f^*(\Upsilon^*(L))$ is ample. With this aim, note that $\Upsilon^*(L)$ satisfies  the numerical assumptions \eqref{E:ipoGKM} by the formula \eqref{E:pull-adj} (with $\beta=0$) using that $L$ is an adjoint line bundle on $\MM_{g,n}^{\ps}$ with  $\alpha_{\irr} \le \frac{10-a}{12}$.
Moreover, $\Upsilon^*(L)$ is positive on all the F-curves of type (6) because of our assumptions on $\Upsilon^*(L)$ and the fact that none of these F-curves has numerical class contained in $\NE(\Upsilon_T)$ by Remark \ref{R:Fcur-T}. 
Therefore, we can apply Lemma \ref{L:GKMmap} in order to conclude that $f^*(\Upsilon^*(L))$ is ample, and we are done. 
 \end{proof}

We are now in position to prove our main result.

\begin{proof}[Proof of Theorem \ref{T:all-one}]

The only if part are trivial, we discuss the if ones. 

Part \eqref{T:all-one1} follows from Proposition \ref{P:F-conjecture}\eqref{P:F-conjecture1}.

Part \eqref{T:all-one2} follows by combining Lemma \ref{L:map-nef}\eqref{L:map-nef1} and Proposition \ref{P:1strata}\eqref{P:1strata1}.

Part \eqref{T:all-one3}: the first assertion follows by applying Lemma \ref{L:map-nef}\eqref{L:map-nef2} and Proposition \ref{P:1strata}\eqref{P:1strata2} to $\Upsilon_*(L)$ which is an adjoint $\bbQ$-divisor on $\MM_{g,n}^{\ps}$ (see Remark \ref{R:comp-adj}\eqref{R:comp-adj1}). The second assertion follows from the first one and  Remark \ref{R:comp-adj}\eqref{R:comp-adj1}. 
\end{proof} 

\subsection{Explicit numerical conditions}

The aim of this subsection is to translate the intersection-theoretic conditions on the adjoint $\bbQ$-divisors in Propositions \ref{P:F-conjecture} and \ref{P:1strata} into explicit numerical inequalities on their coefficients.

 \begin{lemma}\label{L:Fnef}
An adjoint $\bbQ$-divisor $L$ on $\MM_{g,n}$ as in  Definition \ref{D:adj} is F-ample (resp.   F-nef  and the only F-curve on which it is trivial is $C_{\rm ell}$) if and only if 
the following numerical conditions are verified: 
\begin{enumerate}[(i)]
\item \label{L:Fnef1} (for $g\geq 1$) $\displaystyle \alpha_{\irr}>\frac{9-a+\alpha_{1,\emptyset}}{12}$ (resp. $=$); 
\item \label{L:Fnef2} $\alpha_{\irr}<1+\frac{\alpha_{i,I}}{2}$ for any subset $I\subseteq [n]$ and any index $i$ such that $1\leq i \leq g-1$;
\item \label{L:Fnef3} $\alpha_{i,I}+\alpha_{j,J}-\alpha_{i+j,I \cup J} <2$,
 for any disjoint subsets $I,J\subseteq [n]$ and any indices $0\leq i,j$ such that $i+j\leq g-1$;
\item \label{L:Fnef4} $\alpha_{i,I}+\alpha_{j,J}+\alpha_{k,K} - \alpha_{i+j,I\cup J}-\alpha_{i+k,I \cup K}-\alpha_{j+k,J \cup K} + \alpha_{i+j+k,I\cup J \cup K } <  2,$
for any pairwise disjoint subsets $I, J, K\subseteq [n]$ and any indices $0\leq i, j, k$. 
\end{enumerate}
\end{lemma}
\begin{proof}
This follows by intersecting $L$, expressed in the form  \eqref{E:adj-div},  with the F-curves and using Lemma \ref{L:forGKM}: the curve $C_{\rm ell}$ gives rise to \eqref{L:Fnef1}, the F-curves of type (2) and (3) give rise to inequalities that are always satisfied because $L$ is adjoint, the F-curves of type (4) (resp. (5), resp. (6)) give  rise to \eqref{L:Fnef2} (resp. \eqref{L:Fnef3}, resp. \eqref{L:Fnef4}). 
\end{proof}

 Some comments on the numerical conditions appearing in  the above Lemma are in order. 
 
\begin{remark}\label{R:com-ineq}
	\noindent 
	\begin{enumerate}[(i)]
		\item  \label{R:com-ineq1} Condition  \eqref{L:Fnef1} implies that 
		$$
		\begin{aligned}
		\alpha_{\irr} >\frac{9-a+\alpha_{1,\emptyset}}{12}\geq \frac{9-a}{12} & \quad \text{ if } L.C_{\rm ell} >0,  \\
		\alpha_{\irr} =\frac{9-a+\alpha_{1,\emptyset}}{12}\in \left[\frac{9-a}{12}, \frac{10-a}{12}\right] &\quad  \text{ if } L.C_{\rm ell}=0. 
		\end{aligned}
		$$
		
	\item \label{R:com-ineq2}
	
		The inequalities  \eqref{L:Fnef2} are always satisfied if either $\alpha_{\irr}\neq 1$ or $\alpha_{i,I}\neq 0$   for any $[i,I]\in T^*_{g,n}$, and the inequalities \eqref{L:Fnef3} are always satisfied if either $\alpha_{i,I}\neq 1$ for any  $[i,I]\in T^*_{g,n}$ or $\alpha_{i,I}\neq 0$ for any  $[i,I]\in T^*_{g,n}$.
		
	In particular, the inequalities  \eqref{L:Fnef2} are always satisfied if $L\cdot C_{\rm ell}=0$.
		\item \label{R:com-ineq3} The inequalities   \eqref{L:Fnef4}  are always satisfied if 
		$$\alpha_{i,I}>\frac{2}{3} \: \text{ for any } [i,I]\in T^*_{g,n}. $$
	\end{enumerate}
\end{remark}

\begin{lemma}\label{L:F-nef}
Assume that $(g,n)\neq (1,1), (2,0), (1,2)$.
	Let $L$ be an adjoint $\bbQ$-divisor on $\MM_{g,n}^{\ps}$  as in  Definition \ref{D:adj} and fix $T\subset T_{g,n}$. Then $\Upsilon^*L$ is F-nef and the only F-curves on which it is trivial are the ones whose numerical classes belong to $\NE(\Upsilon_T)$ if and only if the following numerical conditions are verified:
	\begin{enumerate}[(i)]
		
		\item \label{L:F-nef1} (for $g\geq 2$)  $\alpha_{\irr} \ge \frac{7-a}{10}$ with equality iff $\irr \in T$;
		
		\item \label{L:F-nef2} 
		\begin{enumerate}
			\item  \label{L:F-nef2a}  $\alpha_{i,I}+\alpha_{j,J}-\alpha_{i+j,I \cup J} <2$ for $i+j\leq g-1$, $I\cap J=\emptyset$, 
			\item  \label{L:F-nef2b}  $12\alpha_{\irr} -7 +a\ge \alpha_{i,I} + \alpha_{i+1,I}$ with equality  iff $\{[i,I], [i+1,I]\} \subset T$;
		\end{enumerate}

		\item \label{L:F-nef3} the following inequalities hold for pairwise disjoint $I,J,H\subseteq \{1,\ldots, n\}$ and indices $0\leq i,j,h\leq g$:
		\begin{enumerate}
			\item \label{L:F-nef3a} $\alpha_{i,I}+\alpha_{j,J}+\alpha_{h,H} - \alpha_{i+j,I\cup J}-\alpha_{i+h,I \cup H}-\alpha_{j+h,J \cup H} + \alpha_{i+j+h,I\cup J \cup H } <2$, 
			\item \label{L:F-nef3b} $(\alpha_{i,I}-\alpha_{i+1,I})+(\alpha_{j,J}-\alpha_{j+1,J})+(\alpha_{i+j+1,I\cup J}-\alpha_{i+j,I\cup J}) < 11-(12\alpha_{\irr}+a)$, 
			\item \label{L:F-nef3c} $(\alpha_{i,I}-\alpha_{i+1,I})+(\alpha_{i+2,I}-\alpha_{i+1,I})-\alpha_{2,\emptyset}< 20 -2(12\alpha_{\irr} +a)$,
			\item  \label{L:F-nef3d} $\alpha_{\irr}<
			\begin{cases} 
			\min_{i\leq g-2} \left\{\frac{11-a+\alpha_{i+1,I}-\alpha_{i,I}}{12}, \frac{ 10-a+\frac{\alpha_{2,\emptyset}}{2}}{12}, \frac{\frac{29}{3}-a+\alpha_{2,\emptyset}-\frac{\alpha_{3,\emptyset}}{3}}{12} \right\} & \text{ if } (g,n)\neq (3,0), (4,0)  \\
			 \frac{\frac{19}{2}-a+\frac{3}{4}\alpha_{2,\emptyset}}{12} & \text{ if } (g,n)=(4,0),\\
			\frac{11-a}{12} & \text{ if } (g,n)=(3,0),\\
			\end{cases}$ 
		\end{enumerate}
	\end{enumerate}
	We assume that all the coefficients of the form $\alpha_{k,K}$ appearing in the above inequalities are so that $[1,\emptyset]\neq [k,K]\in T_{g,n}$. 
\end{lemma}
\begin{proof} 
	In order to prove the Lemma, we need to compute the intersection  of the $\bbQ$-divisor $\Upsilon^*L$ with the  $6$ types of F-curves using Lemma \ref{L:forGKM} and
	the expression \eqref{E:pull-adj}, and check that this intersection is non-negative and zero only on the F-curves described in Remark \ref{R:Fcur-T}.
	
	\begin{enumerate}[(1)]
		\item The intersection of $\Upsilon^*(L)$ with $C_{\rm ell}$ is 0.
		\item The intersection $\Upsilon^*(L)\cdot F(\irr)=2-\alpha_{\irr}$  is  always positive since $\alpha_{\irr}\leq 1$. 
		\item The intersection of $\Upsilon^*(L)$ with $F([i,I])$ (assuming $0\leq i \leq g-2$ and $(i,I)\neq (0,\emptyset)$) is equal to 
		$$\Upsilon^*(L)\cdot F([i,I])=
		\begin{cases}
		2-\alpha_{i,I} & \quad \mbox{ if } [i,I] \ne [1,\emptyset], \\
		(11-12\alpha_{\irr}-a) &  \quad \mbox{ if } [i,I]=[1,\emptyset].
                  \end{cases}
                  $$
		Hence, this intersection is always positive in the first case because $\alpha_{i,I}\leq 1$ by assumption, while in the second case (which can occur only for $g\geq 3$) this happens if and only $\alpha_{\irr}< \frac{11-a}{12}$.
	       This gives rise to condition \eqref{L:F-nef3d} for $(g,n)=(3,0)$ and it is implied by  \eqref{L:F-nef3d} for $(g,n)\neq (3,0)$.

		\item  The intersection of $\Upsilon^*(L)$ with $F_s([i,I])$ (assuming $1\leq i \leq g-2$) is equal to 
		$$\Upsilon^*(L)\cdot F_s([i,I])=
		\begin{cases}
		2(2 - \alpha_{\irr})-( 2-\alpha_{i,I}) & \mbox{ if } [i,I] \ne [1,\emptyset], \\
		2(2 - \alpha_{\irr})-(11-12\alpha_{\irr}-a)  & \mbox{ if } [i,I] = [1,\emptyset].
		\end{cases}
		$$ 
		The first intersection is positive if and only if $\alpha_{\irr} < 1+\frac{\alpha_{i,I}}{2}$ which is implied by   \eqref{L:F-nef3d}. 
		The second intersection is non-negative and zero if and only if $F_s([1,\emptyset])\in \NE(\Upsilon_T)$ (which is equivalent to  $ \irr \in T$ and $g\geq 2$ by Remark \ref{R:Fcur-T}) precisely when \eqref{L:F-nef1} is satisfied. 
	
		\item 		
		By requiring that the intersection of $\Upsilon^*(L)$ with the F-curves $F([i,I],[j,J])$ (for $i+j\leq g-1$ and $I\cap J=\emptyset$) is non-negative and zero only on the F-curves 
		contained in $\NE(\Upsilon_T)$, i.e. those of the form $F([i,I],[i+1,I])$ with  $[i,I],[i+1,I]\in T\setminus \{[1,\emptyset]\}$ by Remark \ref{R:Fcur-T},  we end up with the following six inequalities (the last of which occurs only for $(g,n)=(3,0)$), depending on which indices $[i,I], [j,J], [i+j, I \cup J]$ are equal to $[1,\emptyset]$:
		\begin{align*}
		(2-\alpha_{i,I})+(2-\alpha_{j,J}) > 2-\alpha_{i+j,I \cup J}\ &\Leftrightarrow\ \alpha_{i,I}+\alpha_{j,J}-\alpha_{i+j,I \cup J} <2 \\
		& \quad \text{ for } i+j\leq g-1, I\cap J=\emptyset, \\ 
		(11-12\alpha_{\irr}-a) + (2-\alpha_{i,I}) > 2-\alpha_{i+1,I}\ &\Leftrightarrow\ \alpha_{\irr}< \frac{\alpha_{i+1,I}-\alpha_{i,I} +11-a}{12}\\
		&  \quad \text{ for } i\leq g-2, \\
		(2-\alpha_{i,I})+(2-\alpha_{g-i-1,I^c}) \ge 11-12\alpha_{\irr}-a\ &\Leftrightarrow\ 12\alpha_{\irr}-7+a \ge \alpha_{i,I}+\alpha_{g-i-1,I^c} \\ &\mbox{ with equality iff } \{[i,I], [i+1,I]\} \subset T,\\
		(11-12\alpha_{\irr}-a) + (11-12\alpha_{\irr}-a) > 2-\alpha_{2,\emptyset}\  &\Leftrightarrow\ \alpha_{\irr} < \frac{10-a+\frac{\alpha_{2,\emptyset}}{2}}{12}  \quad \text{ for } g\geq 3,\\
		(11-12\alpha_{\irr}-a)+(2-\alpha_{g-2,[n]}) > 11-12\alpha_{\irr}-a\ &\Leftrightarrow\ \alpha_{g-2,[n]}<2,\\
		(11-12\alpha_{\irr}-a)+ (11-12\alpha_{\irr}-a) > 11-12\alpha_{\irr}-a\ &\Leftrightarrow\ \alpha_{\irr}< \frac{11-a}{12}.\\
		\end{align*}
		The  fifth inequality is always satisfied; the first inequality gives rise to condition  \eqref{L:F-nef2a};  the third inequality gives rise to condition  \eqref{L:F-nef2b}; the sixth inequality gives rise to condition \eqref{L:F-nef3d} for $(g,n)=(3,0)$; the fourth inequality (which cannot  occur  for $(g,n)= (3,0)$) gives rise to the second  inequality  in \eqref{L:F-nef3d} in the   case $(g,n)\neq (3,0), (4,0)$, while it is implied by \eqref{L:F-nef3d} in the case $(g,n)=(4,0)$;  the second inequality (which cannot occur for $(g,n)=(3,0)$ and $(g,n)=(4,0)$) gives rise to the first inequalities in \eqref{L:F-nef3d} for $(g,n)\neq (3,0), (4,0)$. 
		
		\item By requiring that the intersection of $\Upsilon^*(L)$ with the F-curves $F([i,I], [j,J], [k,K])$ is positive, we end up with five inequalities 
		depending on how many indices among $\{[i,I], [j,J], [h, H], [g-i-j-h,[n]\setminus\{I\cup J\cup H\}]\}$ are equal to $[1,\emptyset]$. The case where the number of indices equal to $[1,\emptyset]$ is $0$ (resp.\ $1$, resp.\ $2$) gives rise to conditions \eqref{L:F-nef3a}  (resp.\ \eqref{L:F-nef3b}, resp.\ \eqref{L:F-nef3c}). The case where 
		the number of indices equal to $[1,\emptyset]$ is $3$ (which cannot occur for $(g,n)= (3,0), (4,0)$) gives rise to the third inequality in \eqref{L:F-nef3d} in the case  $(g,n)\neq (3,0), (4,0)$. The case where the 
		 number of indices equal to $[1,\emptyset]$ is $4$ (which can occur only for $(g,n)=(4,0)$) gives rise to the  inequality in \eqref{L:F-nef3d} in the case  $(g,n)= (4,0)$. 
		
	\end{enumerate}		
\end{proof}

Let us comment on the numerical conditions appearing in the above Lemma.  
 
\begin{remark}\label{R:simp-ineq}
	\noindent 
	\begin{enumerate}[(i)]
		\item  \label{R:simp-ineq1} Condition  \eqref{L:F-nef3d} of the above Lemma \ref{L:F-nef} implies that 
		$$\alpha_{\irr}<
		\begin{cases} 
		\frac{10.5-a}{12}& \text{ if } (g,n)\neq (3,0),(4,0),\\
		\frac{10.25-a}{12}& \text{ if } (g,n)= (4,0),\\
		\frac{11-a}{12} & \text{ if } (g,n)=(3,0).\\
		\end{cases}
		$$
		\item \label{R:simp-ineq2}
		The inequalities  \eqref{L:F-nef3b}, \eqref{L:F-nef3c} and \eqref{L:F-nef3d} in the above Lemma simplify under suitable assumptions on $\alpha_{\irr}$ (using that all the coefficients $\alpha_{i,I}$ are such that $0\leq \alpha_{i,I}\leq 1$). More precisely, we have that:
			\begin{itemize}
			\item If $\alpha_{\irr}<  \frac{8-a}{12}$ then \eqref{L:F-nef3b}, \eqref{L:F-nef3c} and \eqref{L:F-nef3d} are always satisfied;
			\item If $\alpha_{\irr}<  \frac{9-a}{12}$ then \eqref{L:F-nef3c} and \eqref{L:F-nef3d} are always satisfied;
			\item If $\alpha_{\irr}<  \frac{\frac{28}{3}-a}{12}$ then \eqref{L:F-nef3d} is always satisfied.	
		\end{itemize}
		\item \label{R:simp-ineq3bis} The inequalities   in \eqref{L:F-nef3} and in \eqref{L:F-nef2a} are always satisfied if 
		$$\alpha_{\irr}\leq  \frac{\frac{29}{3}-a}{12}\quad \text{ and } \quad \vert \alpha_{i,I}-\alpha_{j,J}\vert <\frac{1}{3} \: 
		\text{ for any } [i,I], [j,J] \in T_{g,n}^*\setminus \{[1,\emptyset]\}. $$
		\item \label{R:simpl-ineq4} The inequalities \eqref{L:F-nef2a} are always satisfied if either $\alpha_{i,I}\neq 1$ for any  $[1,\emptyset]\neq [i,I]\in T^*_{g,n}$
 or $\alpha_{i,I}\neq 0$ for any  $[1,\emptyset]\neq [i,I]\in T^*_{g,n}$.
		
	\end{enumerate}
\end{remark}

By using the explicit numerical conditions appearing in the above two Lemmas and the main Theorem \ref{T:all-one}, we get the following corollary which describes a certain region inside the polytope of the adjoint $\bbQ$-divisors on $\MM_{g,n}$ on which we can describe the ample models.

\begin{cor}\label{C:dec-poly}
	Assume that ${\rm char}(k) =0$ and that $(g,n)\neq (1,1), (2,0), (1,2)$. 
	Let $L$ be an adjoint $\bbQ$-divisor on $\MM_{g,n}$
	$$
	L=K+\psi+a\lambda+ \alpha_{\irr}\delta_{\irr} + \sum_{[i,I]\in T^*_{g,n}}\alpha_{i,I}\delta_{i,I}
	$$
	such that $\vert \alpha_{i,I} -\alpha_{j,J}\vert <\frac{1}{3}$ for any $[i,I], [j,J] \in T_{g,n}^*$ and such that if $\alpha_{\irr}=1$ then $\alpha_{i,I}>0$ for any $[i,I]\in T_{g,n}^*$. 
	Assume furthermore that  
	\begin{equation}\label{E:inq1}
	\frac{7-a}{10}\leq \alpha_{\irr} \quad (\text{for } g\geq 2), 
	\end{equation}
	\begin{equation}\label{E:inq2}
	 \frac{7-a+\alpha_{i,I}+\alpha_{i+1,I}}{12}\leq \alpha_{\irr} \text{ for any } [i,I], [i+1,I]\in T^*_{g,n}\setminus \{[1,\emptyset]\}.
	\end{equation}
	Then the ample model of $L$ is 
 \begin{itemize}
        \item $\id:\MM_{g,n}\to \MM_{g,n}$ if $\frac{9-a+\alpha_{1,\emptyset}}{12}<\alpha_{\irr}$;
        \item $\Upsilon: \MM_{g,n}\to \MM_{g,n}^{\ps}$ if $\frac{9-a}{12}<\alpha_{\irr}\leq \frac{9-a+\alpha_{1,\emptyset}}{12}$;
        \item $\Upsilon_T:\MM_{g,n}\to \MM_{g,n}^T$ if $\alpha_{\irr}\leq \frac{9-a}{12} $  where $T$ is admissible and it is uniquely determined by  
        $$
        \begin{aligned} 
        & (\text{for } g\geq 2) \quad \irr \in T \Leftrightarrow \text{ equality holds in } \eqref{E:inq1},\\
        & \{[i,I],[i+1,I]\}\subseteq T \Leftrightarrow \text{ equality holds in } \eqref{E:inq2}.
	\end{aligned}
	$$
 \end{itemize}
	
\end{cor}

\begin{proof}

We will distinguish several cases.

$\bullet$ Assume that $\frac{9-a+\alpha_{1,\emptyset}}{12}<\alpha_{\irr}$.

Using the above assumption, together with   $\vert \alpha_{i,I} -\alpha_{j,J}\vert <\frac{1}{3}$ for any $[i,I], [j,J] \in T_{g,n}^*$ and the fact that if $\alpha_{\irr}=1$ then $\alpha_{i,I}>0$ for any $[i,I]\in T_{g,n}^*$,  it follows from Lemma \ref{L:Fnef} that $L$ if F-ample. Then Theorem \ref{T:all-one}\eqref{T:all-one1} implies that the ample model of $L$ is the identity morphism.

$\bullet$ Assume that $\frac{9-a}{12}<\alpha_{\irr} \leq \frac{9-a+\alpha_{1,\emptyset}}{12}$.

Remark \ref{R:comp-adj}\eqref{R:comp-adj1} implies that the ample model of $L$ is the same as the ample model of $\Upsilon_*(L)$. Therefore, using Theorem \ref{T:all-one}\eqref{T:all-one3}, it is enough to check that $\Upsilon_*(L)$ satisfies the inequalities of Lemma \ref{L:F-nef} with $T=\emptyset$. Conditions \eqref{L:F-nef1} and \eqref{L:F-nef2b} are satisfied with strict inequalities because of the assumption $\frac{9-a}{12}<\alpha_{\irr}$. The inequalities \eqref{L:F-nef2a} and \eqref{L:F-nef3a} are satisfied because of the assumptions  $\vert \alpha_{i,I} -\alpha_{j,J}\vert <\frac{1}{3}$. 
The inequality \eqref{L:F-nef3b} is satisfied because
$$(\alpha_{i,I}-\alpha_{i+1,I})+(\alpha_{j,J}-\alpha_{j+1,J})+(\alpha_{i+j+1,I\cup J}-\alpha_{i+j,I\cup J}) <1\leq 11-(12\alpha_{\irr}+a)$$
where we used that  $\vert \alpha_{i,I} -\alpha_{j,J}\vert <\frac{1}{3}$ and that $\alpha_{\irr}\leq \frac{10-a}{12}$. 
The inequality \eqref{L:F-nef3c} is satisfied because 
$$(\alpha_{i,I}-\alpha_{i+1,I})+(\alpha_{i+2,I}-\alpha_{i+1,I})-\alpha_{2,\emptyset} < \frac{2}{3}-\alpha_{2,\emptyset}< 2-2\alpha_{1,\emptyset}\leq 20 -2(12\alpha_{\irr} +a),$$
where  the first inequality follows from $\vert \alpha_{i,I} -\alpha_{j,J}\vert <\frac{1}{3}$, the second inequality follows from $\alpha_{2,\emptyset}> \alpha_{1,\emptyset}-\frac{1}{3}$ and $\alpha_{1,\emptyset}\leq 1$, and the last inequality follows from $ \alpha_{\irr} \leq \frac{9-a+\alpha_{1,\emptyset}}{12}$.
Finally,  the inequalities  \eqref{L:F-nef3d}  that do not involve $\alpha_{2,\emptyset}$ are satisfied because 
$$\alpha_{\irr}\leq \frac{9-a+\alpha_{1,\emptyset}}{12}\leq \frac{10-a}{12}<\frac{11-a-\frac{1}{3}}{12}, $$
while the ones that involve $\alpha_{2,\emptyset}$ are verified by using that $\alpha_{2,\emptyset}>\alpha_{1,\emptyset}-\frac{1}{3}$ and $\alpha_{1,\emptyset}, \alpha_{3,\emptyset}<1$.

$\bullet$ Assume that $\alpha_{\irr}\leq \frac{9-a}{12}(\leq \frac{9-a+\alpha_{1,\emptyset}}{12})$.

Arguing as in the previous case, it is enough to check that $\Upsilon_*(L)$ satisfies the inequalities of Lemma \ref{L:F-nef} with respect to the subset $T$ defined in the statement. Conditions  \eqref{L:F-nef1} and \eqref{L:F-nef2b} are satisfied by the assumptions \eqref{E:inq1} and \eqref{E:inq2} together with the definition of $T$. The inequalities  \eqref{L:F-nef2a} and \eqref{L:F-nef3} are satisfied by Remark \ref{R:simp-ineq}\eqref{R:simp-ineq3bis}.  
\end{proof}

\bibliographystyle{amsalpha}
\bibliography{Library}

\end{document}